\documentclass{amsart}
\usepackage{amsmath,amssymb,amsthm}
\usepackage{amscd}
\usepackage[all]{xy}
\usepackage{enumerate}
\usepackage{hyperref}
\usepackage{bigints}
\usepackage{color}

\theoremstyle{plain}
\newtheorem{theorem}{Theorem}[section]
\newtheorem{corollary}[theorem]{Corollary}
\newtheorem{lemma}[theorem]{Lemma}
\newtheorem{proposition}[theorem]{Proposition}

\theoremstyle{definition}

\newtheorem{example}[theorem]{Example}

\newtheorem{remark}[theorem]{Remark}

\providecommand{\R}{\mathbb{R}}
\providecommand{\N}{\mathbb{N}}
\providecommand{\Z}{\mathbb{Z}}
\providecommand{\D}{\displaystyle}

\newcommand{\Pa}{\partial}

\renewcommand{\Im}{\mathop{\rm Im}\nolimits}

\newcommand{\norm}[1]{\left|\hspace{-.1em}\left| #1 \right|\hspace{-.1em}\right|}
\newcommand{\V}[1]{\vspace{#1}}

\newcommand{\Cl}[1]{\mathop{\rm Cl}\nolimits\left(#1\right)}

\providecommand{\Diff}{\mathop{\rm Diff}\nolimits}
\providecommand{\Ker}{\mathop{\rm Ker}\nolimits}
\providecommand{\supp}{\mathop{\rm supp}\nolimits}
\providecommand{\Int}{\mathop{\rm Int}\nolimits}
\providecommand{\id}{\mathop{\rm id}\nolimits}
\providecommand{\rank}{\mathop{\rm rank}\nolimits}

\title[Stability of non-proper functions]
{Stability of non-proper functions}

\author[K.~Hayano]{Kenta Hayano}
\address{
Department of Mathematics Faculty of Science and Technology, Keio University
Yagami Campus: 3-14-1 Hiyoshi, Kohoku-ku, Yokohama, 223-8522, Japan}
\email{k-hayano@math.keio.ac.jp}

\dedicatory{Dedicated to Professor Takashi Nishimura on the occasion of his 60th birthday. }

\begin{document}

\maketitle

\begin{abstract}

The purpose of this paper is to give a sufficient condition for (strong) stability of non-proper smooth functions (with respect to the Whitney $C^\infty$-topology). 
We show that a Morse function is stable if it is end-trivial at any point in its discriminant, where end-triviality (which is also called local triviality at infinity) is a property concerning behavior of functions around the ends of the source manifolds. 
We further show that a Morse function $f:N\to \R$ is strongly stable (i.e.~there exists a continuous mapping $g\mapsto (\Phi_g,\phi_g)\in\Diff(N)\times \Diff(\R)$ such that $\phi_g\circ g\circ \Phi_g =f$ for any $g$ close to $f$) if (and only if) $f$ is quasi-proper. 
This result yields existence of a strongly stable but not infinitesimally stable function. 
Applying our result on stability, we give a reasonable sufficient condition for stability of Nash functions, and show that any Nash function becomes stable after a generic linear perturbation. 

\end{abstract}

\section{Introduction} 

A smooth mapping $f$ between manifolds is said to be \emph{stable} if for any mapping $g$ sufficiently close to $f$ (with respect to the Whitney $C^\infty$-topology) one can take diffeomorphisms $\Phi$ and $\phi$ of the source and the target manifolds, respectively, so that they satisfy $\phi\circ f\circ \Phi = g$. 
Stable mappings play an important role not only in the theory of singularities of differentiable mappings, but also in recent studies on topology of smooth manifolds (see \cite{SaekiYamamoto} and \cite{GKtrisection}, for example).

In spite of its simple and natural definition, it is in general difficult to check whether a given smooth mapping is stable or not (cf.~\cite[\S.29]{Whitney}). 
The first breakthrough in detecting stability is due to Mather \cite{MatherII,MatherV}. 
In his seminal work, Mather introduced the notion of \emph{infinitesimal stability} and show that infinitesimal stability implies stability for \emph{proper} smooth mappings (\cite{MatherII}, for the definition of infinitesimal stability, see Subsection~\ref{Se:stability maps}).
Mather further introduced two other variants of stability which we call \emph{strong stability} and \emph{local stability}: $f$ is strongly stable if there exists a continuous mapping $g\mapsto (\Phi_g,\phi_g)\in \Diff(N)\times \Diff(P)$ defined on a neighborhood of $f$ such that $\phi_g\circ f\circ \Phi_g = g$ for any $g$ in the neighborhood, while local stability is a local version of infinitesimal stability (we will give the precise definition of local stability in Subsection~\ref{Se:stability maps}). 
It was then shown in \cite{MatherV} that all the four stabilities are equivalent for \emph{proper} smooth mappings. 

Unfortunately, as Mather already pointed out, the four stabilities above are not equivalent for \emph{non-proper} mappings. 
Whereas it is relatively easy to check local or infinitesimal stability of smooth mappings (see \cite[\S.4 and 5]{MatherV}), the problem of detecting (strong) stability is much harder.
As far as the author knows, the only result concerning stability of non-proper mappings is due to Dimca \cite{Dimca}: he gave a necessary and sufficient condition for stability of Morse functions \emph{defined on $\R$} (see Theorem~\ref{T:Dimca stability R}). 
As for strong stability, du Plessis and Vosegaard \cite{duPlessisVosegaard} gave a necessary and sufficient condition for a smooth mapping to be strongly \emph{and} infinitesimally stable.
The purpose of this paper is to give a sufficient condition for (strong) stability of non-proper functions. 

In order to explain the main result of the paper, we will give several definitions. 
An \emph{open neighborhood of the end of $N$} is an open subset $V$ of $N$ whose complement is compact. 
A smooth mapping $f:N\to P$ is said to be \emph{end-trivial at $y\in P$} if there exist an open neighborhood $W\subset P$ of $y$ and an open neighborhood $V$ of the end of $N$ satisfying the following conditions\footnote{Such an $f$ is also said to be \emph{locally trivial at infinity at $y\in P$} in references.
We will not use this term as we would like to avoid repetition of "at".
}: 

\begin{enumerate}

\item 
$f^{-1}(y)\cap V$ does not contain a critical point of $f$.

\item 
There exists a diffeomorphism $\Phi:(f^{-1}(y)\cap V)\times W\to f^{-1}(W)\cap V$ such that $f\circ \Phi$ is the projection to the second component. 

\end{enumerate}
We denote by $\tau (f)$ the set of points at which $f$ is end-trivial. 
For a smooth mapping $f:N\to P$, let $\Sigma(f)\subset N$ be the set of points $x\in N$ with $\rank(df_x) < p$~($=\dim P$) and $\Delta(f) =f(\Sigma(f))$.
A mapping $f$ is said to be \textit{quasi-proper} if there exists a neighborhood $V\subset P$ of $\Delta(f)$ such that the restriction $f|_{f^{-1}(V)}:f^{-1}(V)\to V$ is proper. 

\begin{theorem}\label{T:suff condi stability}
Let $N$ be a smooth manifold without boundaries and $f:N\to \R$ be a Morse function\footnote{A function $f:N\to \R$ is a \emph{Morse function} if $f|_{\Sigma(f)}$ is injective and $f$ has a non-degenerate Hessian at any point in $\Sigma(f)$. 
Note that $f$ is a Morse function if and only if $f$ is locally stable. }. 

\begin{enumerate}

\item 
A function $f$ is stable if $\Delta(f)$ is contained in $\tau(f)$. 

\item 
A function $f$ is strongly stable if and only if $f$ is quasi-proper. 

\end{enumerate}
\end{theorem}

\noindent
We will give two remarks on this theorem. 
First, it is easy to verify that Dimca's condition for $f:\R\to \R$ mentioned above (see Theorem~\ref{T:Dimca stability R}) is equivalent to our sufficient condition $\Delta(f)\subset \tau(f)$. 
In particular, our sufficient condition is also necessary for stability of functions on $\R$. 
Second, du Plessis and Vosegaard \cite{duPlessisVosegaard} already showed that quasi-properness is a necessary condition for strong stability of smooth mappings, and we will indeed prove that a quasi-proper Morse function is strongly stable. 

The paper is organized as follows. 
We will give several definitions in Section~\ref{Se:preliminaries}. 
Section~\ref{Se:proof main thm} is devoted to the proof of Theorem~\ref{T:suff condi stability}. 
In Section~\ref{Se:application} we will give two applications of Theorem~\ref{T:suff condi stability}. 
First, we will explicitly give an example of strongly stable but not infinitesimally stable function. (See Theorem~\ref{T:ex strongly but not inf stable func}. We indeed prove that the function $F(x)=\exp(-x^2)\sin x$ has the desired properties.) 
As far as the author knows, there was no reasonable condition guaranteeing only strong stability (as we mentioned, du Plessis and Vosegaard \cite{duPlessisVosegaard} gave a necessary and sufficient condition for strong \emph{and} infinitesimal stability of smooth mappings), in particular we could not obtain such an example. 
The second application concerns stability of Nash (i.e.~semi-algebraic smooth) functions. 
For a semi-algebraic mapping $f$, the complement of $\tau(f)$ in the target space is called the set of \textit{bifurcation values at infinity}, which we denote by $B_\infty(f)$, and have been studied in the field of algebraic geometry (see \cite{DTDetectingBifValue} and references therein).
Although it is in general difficult to determine $B_\infty(f)$ completely, one can estimate this set by analyzing several larger sets containing $B_\infty(f)$, defined by considering the \textit{Fedoryuk condition} or the \textit{Malgrange condition}, for example. 
Using these estimates and Theorem~\ref{T:suff condi stability}, we will show that the following function is stable for any $k\in \{1,\ldots,n-1\}$ (Example~\ref{Ex:modelMorse stable}): 
\[
G_k(x_1,\ldots,x_n) = \sum_{i=1}^{k} x_i^2 -\sum_{j=k+1}^n x_j^2. 
\]
Note that we can immediately deduce from Mather's result \cite{MatherV} that $G_0$ and $G_n$ are (strongly) stable since these functions are proper (while $G_k$ for $k \in \{1,\ldots,n-1\}$ is not even quasi-proper).  
Relying on the result in \cite{Ichiki}, we will also prove that any Nash function on $\R^n$ becomes stable after generic linear perturbation (Corollary~\ref{T:stable after linear perturbation}).

\section{Preliminaries}\label{Se:preliminaries}

Throughout the paper, we will assume that manifolds are smooth, second countable and do not have boundaries unless otherwise noted. 
For manifolds $N$ and $P$, we denote the set of smooth mappings from $N$ to $P$ by  $C^\infty(N,P)$. 
Let $\Diff(N)\subset C^\infty(N,N)$ be the set of self-diffeomorphisms of $N$ and $C^\infty(N) = C^\infty(N,\R)$.  
%
Let $X,Y$ be topological spaces. 
For a subset $A\subset X$, we denote its (topological) interior and closure by $\Int(A)$ and $\Cl{A}$, respectively. 
A mapping $f:X\to Y$ is \textit{proper} if the preimage $f^{-1}(K)$ of any compact subset $K\subset Y$ is compact.  
Note that properness of $f$ is equivalent to the condition $Z(f)=\emptyset$, where $Z(f)\subset Y$ is the set of \textit{improper points} of $f$, defined as follows: 
\[
Z(f) =\left\{y\in Y~\left|~\exists\{x_n\}:\begin{minipage}[c]{36mm}
sequence in $X$

without cluster points
\end{minipage}\mbox{s.t. }y=\lim_{n\to \infty}f(x_n)  \right.\right\}.
\]
A smooth mapping $g:N\to P$ between manifolds $N$ and $P$ is \textit{quasi-proper} if there exists a neighborhood $V\subset P$ of the discriminant $\Delta(g)$ such that the restriction $g|_{g^{-1}(V)}:g^{-1}(V)\to V$ is proper. 
This condition is equivalent to the condition $Z(g) \cap \Delta(g) =\emptyset$ (see \cite[Corollary 3.2.15]{duPlessisWall}). 

\subsection{Whitney $C^k$-topology}\label{Se:Whitney topology}
For a non-negative integer $k$, we denote the $k$-jet bundle with the source $N$ and the target $P$ by $J^k(N,P)$.  
For a smooth mapping $f\in C^\infty(N,P)$, let $j^kf:N\to J^k(N,P)$ be the $k$-jet extension of $f$. 
For an open subset $U\subset J^k(N,P)$, we define the subset $M(U)\subset C^\infty(N,P)$ as follows: 
\[
M(U) = \left\{f\in C^\infty(N,P)~|~j^kf(N)\subset U \right\}. 
\]
It is easy to see that the family $\mathcal{M}_k=\{M(U)\subset C^\infty(N,P)~|~U\subset J^k(N,P):\mbox{open}\}$ forms a basis for a topology $\tau W^k$ of $C^\infty(N,P)$, which we call the \textit{Whitney $C^k$-topology}.  
We further define the \textit{Whitney $C^\infty$-topology} $\tau W^\infty$ as a topology with an open basis $\bigcup_{k\geq 0} \mathcal{M}_k$. 

In what follows we will explain a neighborhood basis of $\tau W^k$ due to Mather (\cite{MatherII}).
For a smooth mapping $h:\R^n \to \R^p$, $x\in \R^n$ and a positive integer $k$, we define a linear mapping 
\[
D^kh(x):(\R^n)^{\otimes k} \to \R^p
\]
by giving a value of $e_{i_1}\otimes\cdots\otimes e_{i_k}\in (\R^n)^{\otimes k}$ (where $\{e_1,\ldots,e_n\}$ is the standard basis of $\R^n$) as follows: 
\[
D^kh(x)(e_{i_1}\otimes\cdots\otimes e_{i_k}) = \left(\frac{\Pa^k h_1}{\Pa x_{i_1}\cdots \Pa x_{i_k}}(x),\ldots,\frac{\Pa^k h_p}{\Pa x_{i_1}\cdots \Pa x_{i_k}}(x)\right).
\]
Using the operator norm $\norm{D^k h(x)}$ of $D^kh(x)$, we define $\norm{h}_{k,x}$ and $ \norm{h}_{k,X}$ for $X\subset \R^n$ as follows:  
\[
\norm{h}_{k,x} = \norm{h(x)} + \sum_{j=1}^{k}\norm{D^jh(x)} \mbox{ and }\norm{h}_{k,X} = \sup_{x\in X} \norm{h}_{k,x}. 
\]
Note that for a function $f:\R^n\to \R$, $\norm{D^1f(x)}$ is equal to $\norm{df_x} = \sqrt{\sum_{i=1}^{n}\left(\frac{\Pa f}{\Pa x_i}(x)\right)^2}$ and $\norm{D^2f(x)}$ is equal to 
$\sqrt{\sum_{i=1}^{l} m_i\lambda_i^2}$, where $\lambda_1,\ldots,\lambda_l$ are the eigenvalues of the Hessian of $f$ at $x$ and $m_i$ is the multiplicity of $\lambda_i$. 

For a smooth mapping $f\in C^\infty(N,P)$, we take systems of coordinate neighborhoods $\varphi=\{(U_\alpha,\varphi_\alpha)\}_{\alpha\in A}$ and $\psi=\{(V_\alpha, \psi_\alpha)\}_{\alpha\in A}$ of $N$ and $P$, respectively, and a locally finite covering $K=\{K_\alpha\}_{\alpha \in A}$ of $N$ consisting of compact subsets so that $K_\alpha \subset U_\alpha$ and $f(U_\alpha)\subset V_\alpha$ for each $\alpha \in A$. 
Let $\varepsilon =\{\varepsilon_\alpha\}_{\alpha\in A}$ be a system of positive numbers. 
For each $\alpha \in A$, we denote the following subset by $N_k(f,K_\alpha,\varphi_\alpha,\psi_\alpha,\varepsilon_\alpha)\subset C^\infty(N,P)$: 
\[\left\{g\in C^\infty(N,P) ~\left|~\norm{\psi_\alpha\circ f\circ \varphi_\alpha^{-1}-\psi_\alpha\circ g\circ \varphi_\alpha^{-1}}_{k,\varphi_\alpha(K_\alpha)} < \varepsilon_\alpha\right. \right\}. 
\]
We further define $N_k(f,K,\varphi,\psi,\varepsilon) = \bigcap_{\alpha\in A} N_k(f,K_\alpha,\varphi_\alpha,\psi_\alpha,\varepsilon_\alpha)$. 

\begin{theorem}[{\cite[\S.4, Lemma 1]{MatherII}}]\label{T:basis Whitney topology}
For $k<\infty$ the system 
\[
\{N_k(f,K,\varphi,\psi,\varepsilon)\subset C^\infty(N,P)~|~\varepsilon = \{\varepsilon_\alpha\}_{\alpha\in A}:\mbox{system of positive numbers}\}
\]
is a neighborhood basis of $f\in C^\infty(N,P)$ with respect to the topology $\tau W^k$. 
\end{theorem}

\begin{remark}
For a system $\{V_\alpha\}_{\alpha\in A}$, where $V_\alpha\subset C^\infty(K_\alpha,P)$ is an open subset with respect to the topology $\tau W^k$, we define a subset $\cap_\alpha V_\alpha\subset C^\infty(N,P)$ as follows: 
\[
\cap_\alpha V_\alpha = \{f\in C^\infty(N,P)~|~ \forall \alpha \in A, \hspace{.3em} f|_{K_\alpha} \in V_\alpha\}. 
\] 
\noindent
We can easily deduce from Theorem~\ref{T:basis Whitney topology} that the following system is a basis of $\tau W^k$ for $k<\infty$:
\[
\mathcal{N}_k = \left\{\cap_\alpha V_\alpha\subset C^\infty(N,P) ~\left|~\{V_\alpha\}_{\alpha\in A}\mbox{~:~system of open subsets as above} \right.\right\}. 
\]
We can also define a system $\mathcal{N}_\infty$ of subsets of $C^\infty(N,P)$ in a similar manner, yet it is \emph{not} a basis of the topology $\tau W^\infty$ but produces a stronger topology of $C^\infty(N,P)$, which is called the \emph{very strong topology} in \cite{duPlessisVosegaard}.

\end{remark}

\subsection{Stability of smooth mappings}\label{Se:stability maps}

A smooth mapping $f\in C^\infty(N,P)$ is \textit{stable} if there exists an open neighborhood $\mathcal{U}\subset C^\infty(N,P)$ of $f$ (with respect to the topology $\tau W^\infty$) and (not necessarily continuous) mappings $\Theta:\mathcal{U}\to \Diff(N)$ and $\theta:\mathcal{U}\to \Diff(P)$ such that $\theta(g)\circ g\circ\Theta(g) = f$ for $g\in \mathcal{U}$. 
A smooth mapping $f\in C^\infty(N,P)$ is \textit{strongly stable} if we can further make $\Theta$ and $\theta$ above continuous (with respect to the topologies $\tau W^\infty$). 
For a vector bundle $E$ on $N$, we denote the set of smooth sections of $E$ by $\Gamma(E)$, which is a $C^\infty(N)$-module. 
We define a $C^\infty(N)$-module homomorphism $tf: \Gamma(TN) \to \Gamma(f^\ast TP)$ and a $C^\infty(P)$-module homomorphism $\omega f:\Gamma(TP)\to \Gamma(f^\ast TP)$ as follows: 
\[
tf(\xi) = df\circ \xi\mbox{ and }\omega f(\eta) = \eta \circ f. 
\]
A smooth mapping $f\in C^\infty(N,P)$ is \textit{infinitesimally stable} if the following equality holds: 
\[
\Gamma(f^\ast TP) = tf(\Gamma(TN)) + \omega f (\Gamma(TP)). 
\]
For a subset $S\subset N$ we denote the set of germs of sections of a vector bundle $E$ on $N$ at $S$ by $\Gamma(E)_S$. 
The homomorphism $tf$ (resp.~$\omega f$) induces a homomorphism from $\Gamma(TN)_S$ (resp.~$\Gamma(TP)_{f(S)}$) to $\Gamma(f^\ast TP)_S$ in the obvious way. 
A smooth mapping $f\in C^\infty(N,P)$ is \textit{locally stable} if the following equality holds for any $y\in \Delta(f)$ and $S\subset f^{-1}(y)$ with $\sharp(S)\leq \dim P+1$: 
\[
\Gamma(f^\ast TP)_S = tf(\Gamma(TN)_S) + \omega f (\Gamma(TP)_y). 
\]

As we noted in the introduction, all the four stabilities above are equivalent for a proper mapping $f\in C^\infty(N,P)$. 
In the rest of the subsection we will briefly review known results on stabilities for general (non-proper) mappings. 
We can immediately deduce from the definitions that strong stability (resp.~infinitesimal stability) implies stability (resp.~local stability). 
Mather \cite{MatherV} showed that stability implies local stability and a smooth mapping $f\in C^\infty(N,P)$ is infinitesimally stable if and only if it is locally stable and $f|_{\Sigma(f)}$ is proper. 
It was shown in \cite{duPlessisVosegaard} that any strongly stable mapping is quasi-proper. 
Furthermore, Dimca \cite{Dimca} gave a necessary and sufficient condition for stability of a function $f\in C^\infty(\R,\R)$: 

\begin{theorem}[\cite{Dimca}]\label{T:Dimca stability R}
A locally stable function $f\in C^\infty(\R,\R)$ is stable if and only if the intersection $\Delta(f) \cap (Z(f|_{\Sigma(f)})\cup L(f))$ is empty, where $L(f)$ is defined as follows: 
\[
L(f) = \left\{y\in \R~\left|~y = \lim_{x\to \pm \infty} f(x) \right.\right\}. 
\]
\end{theorem}

\noindent
As we mentioned in the introduction, the necessary and sufficient condition in this theorem (i.e.~$\Delta(f) \cap (Z(f|_{\Sigma(f)})\cup L(f))=\emptyset$) is equivalent to the condition $\Delta(f)\subset\tau(f)$ (the sufficient condition for stability in the main theorem). 
We can indeed show that the complement $\R\setminus \tau(f)$ is equal to $Z(f|_{\Sigma(f)})\cup L(f)$ for $f\in C^\infty(\R,\R)$.

\section{A sufficient condition for stability of smooth functions}\label{Se:proof main thm}

In this section we will first prove (1) of Theorem~\ref{T:suff condi stability}. 
In the proof, for a given function $f$ satisfying the assumption in (1), we will construct several diffeomorphisms so that the composition of them with a function close to $f$ coincides with $f$.  
We will then observe that the algorithm for constructing diffeomorphisms also guarantees strong stability of quasi-proper Morse functions.  
Throughout the paper, we denote the open ball in the Euclidean space with radius $r$ centered at the origin by $B(r)$.

\begin{lemma}\label{T:arbitrariness nbd end}

A mapping $f\in C^\infty(N,P)$ is end-trivial at $y\in P$ if and only if, for any compact subset $K\subset N$, there exist an open neighborhood $V$ of the end of $N$ with $K\subset N\setminus V$ and an open neighborhood $W\subset P$ of $y$ satisfying the conditions (1) and (2) in the definition of end-triviality. 

\end{lemma}

\noindent
In other words, one can take $V$ in the definition of end-triviality "as small as desired". 
As the "if" direction of the lemma is obvious, we will prove the other direction below: 

\begin{proof}
Since we will discuss local behavior of $f$ around the fiber $f^{-1}(y)$, we can assume $P=\R^p$ and $y=0$ without loss of generality. 
By the assumption, there exist an open neighborhood $V'$ of the end of $N$, an open neighborhood $W'\subset \R^p$ of the origin and $\Phi:(f^{-1}(0)\cap V')\times W'\to f^{-1}(W')\cap V'$ satisfying the conditions in the definition of end-triviality. 
By composing a self-diffeomorphism $(p_1\circ \Phi^{-1}|_{f^{-1}(0)\cap V'})\times \id$ of $(f^{-1}(0)\cap V')\times W'$ to $\Phi$ if necessary, we can assume that $\Phi|_{(f^{-1}(0)\cap V')\times \{0\}}$ is the projection to the first component. 

Let $K\subset N$ be a compact subset. 
We take a proper function $h:N\to \R$ and let $L_m = \{x\in N~|~ \left|h(x)\right|\leq m\}$ for $m\geq 0$. 
Since $N\setminus V'$ and $K$ are both compact, there exists $m_0\in\N$ with  $(N\setminus V')\cup K\subset \Int(L_{m_0})$. 
We put $K_1 = L_{m_0}$ and $K_2=L_{m_0+1}$.
Note that these are both compact and $K_1\subset \Int(K_2)$. 
In what follows, we will show that the following conditions hold for a sufficiently small $\delta >0$: 

\begin{enumerate}

\item
$B(\delta)\subset W'$. 
\item
$\Phi((f^{-1}(0)\setminus \Int(K_2))\times \Cl{B(\delta)})\cap K_1=\emptyset$. 

\item
$\Phi((f^{-1}(0)\cap K_1\cap V')\times \Cl{B(\delta)})\subset \Int(K_2)$. 

\end{enumerate}

\noindent
The condition (1) obviously holds for a small $\delta>0$. 
Suppose first that there does not exist $\delta$ satisfying the condition (2). 
We can then take a point $x_m\in \Phi((f^{-1}(0)\setminus \Int(K_2))\times \Cl{B(1/m)})\cap K_1$ for any $m\in \N$. 
As $K_1$ is compact, $\{x_m\}_{m\in \N}$ has a cluster point $x\in K_1$. 
Since $\norm{f(x_m)}\leq 1/m$, $f(x)$ is equal to $0$. 
We can deduce from this observation, together with closedness of $f^{-1}(0)\setminus \Int(K_2)$, that $x$ is contained in $\Phi((f^{-1}(0)\setminus \Int(K_2))\times \{0\})$, which is equal to $f^{-1}(0)\setminus \Int(K_2)$ by the assumption on $\Phi$. 
This contradicts the condition $K_1\subset \Int(K_2)$. 
Suppose next that there is no $\delta$ satisfying the condition (3). 
For any $m\in \N$, we can take $(w_m,t_m)\in (f^{-1}(0)\cap K_1\cap V')\times \Cl{B(1/m)}$ so that the image $\Phi(w_m,t_m)$ is not contained in $\Int(K_2)$. 
By the assumption on $\Phi$, $\Phi(w_m,0)$ is equal to $w_m$, in particular it is contained in $K_1$. 
By the intermediate value theorem, there exists $s_m\in B(\norm{t_m})$ such that $h(\Phi(w_m,s_m)) = m_0+1/2$ for any $m\in \N$. 
Since $\Phi(w_m,s_m)$ is contained in the compact subset $h^{-1}(m_0+1/2)$, the sequence $\{\Phi(w_m,s_m)\}_{m\in \N}$ has a cluster point $x'\in h^{-1}(m_0+1/2)$. 
As $f\circ \Phi(w_m,s_m) = s_m$ and $\norm{s_m} \leq 1/m$, $x'$ is contained in $f^{-1}(0)\setminus K_1\subset f^{-1}(0)\setminus V'$. 
Thus, we can take a neighborhood $U\subset (f^{-1}(0)\setminus K_1)\times W'$ of $(x',0)$ so that $\Phi(U)$ is a neighborhood of $x'$ and contained in the open set $\Int(K_2)\setminus K_1$. 
In particular, $\Phi(w_m,s_m)$ is contained in $\Phi(U)$ for a sufficiently large $m$, which contradicts that $\Phi$ is a homeomorphism (especially injective). 

Using $\delta>0$ satisfying the conditions above, we define $L$ and $V$ as follows: 
{\allowdisplaybreaks
\begin{align*}
&L = p_1\circ \Phi^{-1}\left(f^{-1}(\Cl{B(\delta)})\cap K_2\setminus \Int(K_1)\right), \\
&V = N\setminus \left(\Phi(L\times \Cl{B(\delta)})\cup K_1\right). 
\end{align*}
}%
It is easy to check that $L$ is compact and $V$ is an open neighborhood of the end of $N$. 
In what follows, we will show that $V$ and $W = B(\delta)$, together with the restriction of $\Phi$, satisfy the desired conditions. 

As $f^{-1}(0)\cap V$ is a subset of $f^{-1}(0)\cap V'$, it does not contain any critical point of $f$. 
We first show that the image of $(f^{-1}(0)\cap V)\times B(\delta)$ by $\Phi$ is contained in $f^{-1}(B(\delta))\cap V$. 
Let $(x,q)\in (f^{-1}(0)\cap V)\times B(\delta)$. 
Since $f\circ \Phi=p_2$, $\Phi(x,q)$ is contained in $f^{-1}(B(\delta))$. 
As $\Phi|_{(f^{-1}(0)\cap V')\times \{0\}}$ is assumed to be the projection to the first component, the following holds: 
\[
f^{-1}(0)\cap V=f^{-1}(0)\setminus (K_1\cup L) \mbox{ and } f^{-1}(0)\cap K_2\setminus \Int(K_1)\subset f^{-1}(0)\cap L. 
\] 
Thus, $f^{-1}(0)\cap V$ is contained in $f^{-1}(0)\setminus K_2$, and $\Phi(x,q)$ is not contained in $\Phi(L\times \Cl{B(\delta)})$ and $K_2\setminus \Int(K_1)$. 
Furthermore, since $x$ is in $f^{-1}(0)\cap V \subset f^{-1}(0)\setminus K_2$, $\Phi(x,q)$ is not contained in $K_1$ by the condition (2) above. 
We can eventually conclude that $\Phi(x,q)$ is contained in $f^{-1}(B(\delta))\cap V$. 

We next show that the restriction $\Phi:(f^{-1}(0)\cap V)\times B(\delta) \to f^{-1}(B(\delta))\cap V$ is surjective, and thus a diffeomorphism. 
Let $w\in f^{-1}(B(\delta))\cap V$. 
Since $f^{-1}(B(\delta))\cap V$ is contained in $f^{-1}(W')\cap V'$, there exists $w'\in f^{-1}(0)\cap V'$ with $w = \Phi(w',f(w))$. 
As $w$ is not contained in $\Phi\left(L\times \Cl{B(\delta)}\right)$, $w'$ is not contained in $L$. 
By the condition (3) above, $w$ would be contained in $K_2$ if $w'$ is in $K_1$, but this contradicts the condition that $w$ is not contained in $\Phi\left(L\times \Cl{B(\delta)}\right)$ since it contains $ f^{-1}(\Cl{B(\delta)})\cap K_2\setminus \Int(K_1)$. 
Therefore, $w'$ is contained in $f^{-1}(0)\setminus (K_2\cup L) = f^{-1}(0)\cap V$, completing the proof of Lemma~\ref{T:arbitrariness nbd end}.  
\end{proof}

Let $N$ be an $n$-dimensional manifold and $f\in C^\infty(N)$. 
Suppose that $f$ is locally stable and $\Delta(f)\subset \tau(f)$. 
We put $\Sigma=\Sigma(f)$ and $\Delta=\Delta(f)$.

\begin{lemma}\label{T:countable discrete Delta}

The set $\Delta\subset \R$ is countable and discrete. 

\end{lemma}

\begin{proof}
Since $f$ is locally stable, $f$ has a non-degenerate Hessian at any point in $\Sigma$. 
We can thus deduce from the Morse lemma that $\Sigma\subset N$ is closed and discrete.
In particular, the intersection $\Sigma\cap K$ is finite for any compact subset $K\subset N$. 
Since $N$ is covered by a countable system of compact subsets, the set $\Sigma$ and $\Delta=f(\Sigma)$ are countable. 

By the definition of end-triviality, for any $y\in \tau(f)$, there exist a compact subset $K\subset N$ and a neighborhood $W\subset \R$ of $y$ such that $f^{-1}(W)\setminus K$ does not contain any critical point of $f$, in particular the number of critical points in $f^{-1}(W)$ is finite.  
Since $\Delta$ is contained in $\tau(f)$ by the assumption, each point $y\in \Delta$ has a neighborhood $W\subset \R$ with $W\cap \Delta$ finite, and thus $\Delta$ is discrete.
\end{proof}

\noindent
We put $\Delta = \{y_i\}_{i\in \mathcal{I}}$, where $\mathcal{I}\subset \N$, and let $x_i\in \Sigma$ be the critical point of $f$ with $f(x_i)=y_i$.
Let $K_0=\emptyset$ and we take a sequence $\{K_m\}_{m\in \N}$ of codimension-$0$ compact submanifolds of $N$ satisfying the following conditions: 

\begin{itemize}

\item 
$\D\bigcup_{m\in \N} K_m = N$. 

\item 
$K_m\subset \Int(K_{m+1})$ for any $m\in \N$. 

\item 
$x_i\in \Int(K_i)\setminus K_{i-1}$ for any $i\in \mathcal{I}$. 

\end{itemize}
By the assumption, $y_i$ is contained in $\tau(f)$ for each $i\in \mathcal{I}$.
By Lemma~\ref{T:arbitrariness nbd end}, we can take a neighborhood $V_i$ of the end of $N$ and $\nu_i>0$ satisfying the following conditions\footnote{We denote by $f^{-1}(y_i-2\nu_i,y_i+2\nu_i)$ the preimage of the open interval $(y_i-2\nu_i,y_i+2\nu_i)\subset \R$ by $f$. Although it should be denoted by $ f^{-1}\left((y_i-2\nu_i,y_i+2\nu_i)\right)$, we will omit a pair of parentheses throughout the paper for simplicity.}:

\begin{itemize}

\item
$K_i\subset N\setminus V_i$. 

\item
There exists a diffeomorphism
\[
\Phi_i:(f^{-1}(y_i)\cap V_i)\times (y_i-2\nu_i,y_i+2\nu_i)\to f^{-1}(y_i-2\nu_i,y_i+2\nu_i)\cap V_i
\]
such that the composition $f\circ \Phi_i$ is the projection to the second component.

\item
$f^{-1}(y_i-2\nu_i,y_i+2\nu_i)\cap V_i=\emptyset$ if $y_i$ is not contained in $Z(f)$. 

\end{itemize}

\noindent
Since the complement $N\setminus V_i$ is compact, $N\setminus V_i$ is contained in $K_{d(i)-1}$ for some $d(i)\in \N$, which is larger than $i$.
Since the set $\Delta$ is discrete by Lemma~\ref{T:countable discrete Delta} and $f$ is locally stable, by replacing $\nu_i$'s with smaller ones and restricting $\Phi_i$ accordingly, we can assume the following conditions without loss of generality:  

\begin{itemize}

\item
$\nu_i <1/4$, 

\item
$(y_i-2\nu_i,y_i+2\nu_i)\cap (y_j-2\nu_j,y_j+2\nu_j)=\emptyset$ for $i\neq j$. 

\item
There exists a relatively compact coordinate neighborhood $(U_i,\varphi_i)$ around $x_i$ such that $\varphi_i(U_i)=B(\nu_i)$, $f\circ\varphi_i^{-1}(w_1,\ldots,w_n) = \pm w_1^2\pm \cdots \pm w_n^2 + y_i$, and $U_i$ is contained in $\Int (K_i)\setminus K_{i-1}$.
Note that $U_i$ is also contained in $f^{-1}(y_i-\nu_i/2,y_i+\nu_i/2)$ since $\nu_i<1/4$. 

\end{itemize}

In order to construct diffeomorphisms of $N$ and $P$ for each mapping close to $f$, and guaranteeing continuity of the resulting mappings (to $\Diff(N)$ and $\Diff(\R)$), we need to take a suitable system of coordinate neighborhoods of $N$ as follows:  

\begin{lemma}\label{T:existence system cdnt nbhd}

There exist a system of coordinate neighborhoods $\{(U_\alpha,\varphi_\alpha)\}_{\alpha\in A}$ of $N$ satisfying the following conditions: 

\begin{enumerate}[(U1)]

\item 
$U_\alpha$ is relatively compact.

\item 
The index set $A$ contains $\mathcal{I}$ (that is, $(U_i,\varphi_i)$ taken above is in the system). 

\item 
If $U_\alpha$ is contained in $f^{-1}(y_i-\nu_i,y_i+\nu_i)$ for some $i\in \mathcal{I}$, either of the followings holds: 

\begin{itemize}

\item[\textit{(U3-1)}]
$U_\alpha \subset \Int (K_{d(i)})$.

\item[\textit{(U3-2)}]
There exists a coordinate neighborhood $(W_\alpha,\phi_\alpha)$ of $f^{-1}(y_i)\cap V_i$ (which is a manifold since it contains no critical point of $f$) such that $U_\alpha$ is equal to $\Phi_i\left(W_\alpha \times (y_i-\nu_i,y_i+\nu_i)\right)$ and $\varphi_\alpha=(\phi_\alpha\times \id)\circ \Phi_i^{-1}$.

\end{itemize}

\end{enumerate}

\noindent
Let $M$ be the union $\cup_{i\in \mathcal{I}} f^{-1}(y_i-\nu_i,y_i+\nu_i)$. 

\begin{enumerate}[(U1)]
\setcounter{enumi}{3}

\item 
If $U_\alpha$ is not contained in $M$, then for any $i\in \mathcal{I}$ with $U_\alpha \cap f^{-1}([y_i-\nu_i/2,y_i+\nu_i/2])\neq \emptyset$, $U_\alpha$ is contained in $K_i\setminus U_{i}$.

\item 
For each $i\in \mathcal{I}$, there exist only finitely many $\alpha$'s satisfying the condition $U_\alpha \cap N_i\neq \emptyset$, where $N_i=f^{-1}([y_i-\nu_i/2,y_i+\nu_i/2])\cap K_{d(i)+1}\setminus \Int(K_{d(i)})$. 

\item 
For any $i\in \mathcal{I}$ and $\alpha\neq i$, $U_\alpha\cap \varphi_i^{-1}(\Cl{B(\nu_i/2)})=\emptyset$. 

\end{enumerate}

\end{lemma}

\begin{proof}
The set $N_i$ is compact and contained in $f^{-1}(y_i-\nu_i,y_i+\nu_i)\cap V_i$, that is, the target of $\Phi_i$, where $N_i=f^{-1}([y_i-\nu_i/2,y_i+\nu_i/2])\cap K_{d(i)+1}\setminus \Int(K_{d(i)})$, as defined in the condition (U5). 
Let $Z_i\subset f^{-1}(y_i)\cap V_i$ be the image of this compact set by $p_1\circ \Phi_i^{-1}$ (where $p_1$ is the projection to the first component), which is also compact. 
We take a system $\{(W_\beta,\phi_\beta)\}_{\beta \in B_i}$ of coordinate neighborhoods of the manifold $f^{-1}(y_i)\cap V_{i}$ so that $\Cl{W_\beta}\subset f^{-1}(y_i)\cap V_{i}$ is compact for each $\beta \in B_i$, and $Z_i\cap W_\beta\neq \emptyset$ for only finitely many $\beta \in B_i$. 
We can take such a system by first taking a finite system $\{(W_{\beta_m},\phi_{\beta_m})\}_{m=1}^k$ covering $Z_i$, and then taking a coordinate neighborhood for each point in the complement of $\cup_{m=1}^k W_{\beta_m}$ so that it is away from $Z_i$.
For $i\in \mathcal{I}$ and $\beta \in B_i$, we put $U_\beta=\Phi_i\left(W_{\beta} \times (y_i-\nu_i,y_i+\nu_i)\right)$ and $\varphi_\beta=(\phi_\beta\times p_1) \circ \Phi_i^{-1}|_{U_\beta}$.
Note that $\{U_\beta\}_{\beta\in B_i}$ is a covering of $f^{-1} (y_i-\nu_i,y_i+\nu_i) \setminus V_i$. 

Let $C_i = f^{-1}(y_i-\nu_i,y_i+\nu_i)\cap K_{d(i)-1}\setminus U_{i}$. 
For any $x\in C_i$, we take a relatively compact coordinate neighborhood $(U_x,\varphi_x)$ around $x$ so that $U_x$ is contained in the following set: 
\[
\Int(K_{d(i)})\cap f^{-1}(y_i-\nu_i,y_i+\nu_i)\setminus \varphi_i^{-1}(\Cl{B(\nu_i/2)}).
\]
Note that $\{U_x\}_{x\in C_i}$ is a covering of $C_i$. 
Thus, the preimage $f^{-1}(y_i-\nu_i,y_i+\nu_i)$ is contained in the following union: 
\[
U_i \cup \left(\bigcup_{\beta\in B_i} U_\beta\right) \cup \left(\bigcup_{x\in C_i} U_x\right). 
\]

Suppose that $x\in N$ is not contained in $f^{-1}(\Cl{\Delta})\cup M$, where $M$ is the union $\cup_{i\in \mathcal{I}} f^{-1}(y_i-\nu_i,y_i+\nu_i)$ (as defined in the lemma). 
The following functions have the minimum values: 
{\allowdisplaybreaks
\begin{align*}
&h_+:\Cl{\Delta}\cap [f(x),\infty) \to \R, \hspace{.5em} h_+(y) = y-f(x), \\
&h_-:\Cl{\Delta}\cap (-\infty,f(x)] \to \R, \hspace{.5em} h_-(y) = f(x)-y.
\end{align*}
}%
We denote the minimizer of $h_\pm$ by $y_\pm \in \R$. 
We take $\tilde{y}_\pm \in \R$ as follows:
\[
\tilde{y}_\pm = \begin{cases}
y_i \mp \nu_i/2 & (y_\pm =y_i \in \Delta \mbox{ for some }i\in \mathcal{I}),\\
y_\pm & (y_\pm \not\in \Delta). 
\end{cases}
\] 
By the assumption, $\tilde{y}_+$ (resp.~$\tilde{y}_-$) is strictly larger (resp.~smaller) than $f(x)$. 
Thus, the open interval $(\tilde{y}_-,\tilde{y}_+)\subset \R$ contains $f(x)$. 
Furthermore, the interval $(\tilde{y}_-,\tilde{y}_+)$ does not intersect with $[y_j-\nu_j/2,y_j+\nu_j/2]$ for any $j\in \mathcal{I}$. 
To prove this, we assume that $(\tilde{y}_-,\tilde{y}_+)\cap [y_j-\nu_j/2,y_j+\nu_j/2]$ contains $y'\in \R$. 
First, $y'$ is not equal to $f(x)$ since $f(x)$ is not in $[y_j-\nu_j/2,y_j+\nu_j/2]$. 
Suppose that $y'$ is greater than $f(x)$. 
Since $y_+$ is the minimizer of $h_+$, $y_j$ is either less than $f(x)$ or greater than $y_+$($\geq \tilde{y}_+$). 
If the value $y_j$ is less than $f(x)$, the interval $[y_j-\nu_j/2,y_j+\nu_j/2]$ would contain both $y_j<f(x)$ and $y'>f(x)$, and thus contain $f(x)$, contradicting the condition $x\not\in M$. 
Hence, $y_j$ is greater than $\tilde{y}_+$, and the interval $[y_j-\nu_j/2,y_j+\nu_j/2]$ contains $y'<\tilde{y}_+$, $y_j >\tilde{y}_+$, and thus $\tilde{y}_+$. 
Note that $\tilde{y}_+$ is either $y_i-\nu_i/2$ (if $y_+=y_i$ for some $i\in \mathcal{I}$) or $y_+$ (if $y_+\not\in \Delta$). 
Since $[y_j-\nu_j/2,y_j+\nu_j/2]\cap [y_i-\nu_i/2,y_i+\nu_i/2]$ is empty, $y_i-\nu_i/2$ is not in $[y_j-\nu_j/2,y_j+\nu_j/2]$, and thus $\tilde{y}_+=y_+ \not\in \Delta$. 
However, the interval $(y_j-2\nu_j,y_j+2\nu_j)$ contains only one critical value $y_j$ of $f$ and $\tilde{y}_+$ is a cluster point of $\Delta$, $y_+$ is not in $[y_j-\nu_j/2,y_j+\nu_j/2]$, either, leading to a contradiction. 
Therefore, $y'$ is not greater than $f(x)$. 
We can also show that $y'$ is not smaller than $f(x)$ in a similar manner, completing the proof that $(\tilde{y}_-,\tilde{y}_+)$ does not intersect with $[y_j-\nu_j/2,y_j+\nu_j/2]$ for any $j\in \mathcal{I}$. 
We take a relatively compact coordinate neighborhood $(U_x,\varphi_x)$ around $x$ above so that $\Cl{U_x}$ is contained in $f^{-1}(\tilde{y}_-,\tilde{y}_+)$. 
Note that $U_x$ does not intersect with $f^{-1}([y_i-\nu_i/2,y_i+\nu_i/2])$ for any $i\in \mathcal{I}$. 

Suppose that $x$ is in the closed subset $f^{-1}(\Cl{\Delta})\setminus M$. 
We take $c\in \N$ so that $x$ is contained in $K_c\setminus K_{c-1}$. 
Let $j_1,\ldots, j_e\in \mathcal{I}$ be the numbers in $\mathcal{I}$ which are less then or equal to $c+1$. 
The union $\cup_{m=1}^e [y_{j_m}-\nu_{j_m}/2,y_{j_m}+\nu_{j_m}/2]$ is a closed subset of $\R$ and does not contain $f(x)$ by the assumption. 
We take an open interval $J\subset \R$ so that $f(x)\in J$ and $J$ is away from $\cup_{m=1}^e [y_{j_m}-\nu_{j_m}/2,y_{j_m}+\nu_{j_m}/2]$. 
We also take a relatively compact coordinate neighborhood $(U_x,\varphi_x)$ around $x$ so that $U_x$ is contained in the following open set: 
\[
f^{-1}(J) \cap \Int(K_{c+1})\setminus K_{c-1}. 
\]
For each $m\in\N$, the set $f^{-1}(\Cl{\Delta})\cap K_m\setminus (M\cup \Int(K_{m-1}))$ is compact. 
Thus, there exists a finite subset $D_m\subset f^{-1}(\Cl{\Delta})\cap K_m\setminus (M\cup \Int(K_{m-1}))$ such that $\{U_x\}_{x\in D_m}$ is a covering of $f^{-1}(\Cl{\Delta})\cap K_m\setminus (M\cup \Int(K_{m-1}))$. 
We denote the union $\cup_{m\in \N} D_m$ by $D$.
Note that $\{U_x\}_{x\in D}$ is a covering of $f^{-1}(\Cl{\Delta})\setminus M$. 
We can eventually define the index set $A$ for a desired system as follows: 
\[
A = \mathcal{I} \sqcup \left(\bigcup_{i\in \mathcal{I}} B_i\right)\sqcup \left(\bigcup_{i\in \mathcal{I}} C_i\right)\sqcup \left(N\setminus \left(f^{-1}(\Cl{\Delta})\cup M\right)\right)\sqcup D. 
\]
One can deduce from what we have observed that $\{U_\alpha\}_{\alpha \in A}$ is a covering of $N$. 
In what follows, we will check that the system $\{(U_\alpha,\varphi_\alpha)\}_{\alpha\in A}$ satisfies the conditions (U1)--(U6) in Lemma~\ref{T:existence system cdnt nbhd}. 

The condition (U2) obviously holds. 
The closure $\Cl{U_\beta}$ is compact for each $\beta \in B_i$ as it is homeomorphic to $\Cl{W_\beta}\times [y_i-\nu_i,y_i+\nu_i]$. 
Thus, $U_\beta$ is relatively compact. 
Since the other coordinate neighborhoods were taken so that these are relatively compact, the condition (U1) holds for $\{(U_\alpha,\varphi_\alpha)\}_{\alpha\in A}$.  

In order to show that the condition (U3) holds, we assume that $U_\alpha$ is contained in $f^{-1}(y_i-\nu_i,y_i+\nu_i)$ for some $i$. 
The index $\alpha$ is $i$, or contained in either $B_i$ or $C_i$. 
Since $U_i$ is contained in $K_i\subset K_{d(i)}$, the condition (U3-1) holds if $\alpha = i$. 
Furthermore, we can deduce from the construction above that $U_\alpha$ satisfies the condition (U3-1) (resp.~(U3-2)) if $\alpha$ is in $C_i$ (resp.~$B_i$). 
We can thus conclude that (U3) holds for $\{(U_\alpha,\varphi_\alpha)\}_{\alpha\in A}$. 

Suppose that $U_\alpha$ is not contained in $M$.
The index $\alpha$ is in either the complement $N\setminus \left(f^{-1}(\Cl{\Delta})\cup M\right)$ or $D$. 
The neighborhood $U_\alpha$ is away from $f^{-1}([y_i-\nu_i/2,y_i+\nu_i/2])$ for any $i\in \mathcal{I}$ if $\alpha \in N\setminus \left(f^{-1}(\Cl{\Delta})\cup M\right)$. 
We assume that $\alpha=x \in D$ and $U_\alpha\cap f^{-1}([y_i-\nu_i/2,y_i+\nu_i/2])\neq \emptyset$. 
Let $c\in \N$ be the number with $x\in K_c\setminus K_{c-1}$. 
By the construction above, $i$ is greater than $c+1$. 
Since $U_x$ is contained in $\Int(K_{c+1})$, it is also contained in $K_{i-1}$. 
Since $U_{i}$ is contained in $\Int(K_i)\setminus K_{i-1}$, $U_x$ is away from $U_{i}$, concluding that (U4) holds for $\{(U_\alpha,\varphi_\alpha)\}_{\alpha\in A}$.

Let $i,j\in \mathcal{I}$ be numbers with $i\neq j$. 
The intersection $N_i\cap U_{j}$ is empty. 
For $\alpha \in B_j\cap C_j$, $N_i\cap U_\alpha$ is also empty since $U_\alpha$ is contained in $f^{-1}(y_j-\nu_j,y_j+\nu_j)$. 
By the construction, there are only finitely many $\alpha \in B_i$ with $N_i\cap U_\alpha\neq \emptyset$ (i.e.~$\beta_1,\ldots,\beta_k\in B_i$ given above). 
For $\alpha\in C_i$, the intersection $N_i\cap U_\alpha$ is empty since $U_\alpha$ is contained in $\Int(K_{d(i)})$. 
Since $N_i$ is not contained in $K_i$, we can deduce from the condition (U4) that $U_\alpha \cap N_i=\emptyset$ if $U_\alpha \not\subset M$, that is, $\alpha$ is in $\left(N\setminus \left(f^{-1}(\Cl{\Delta})\cup M\right)\right)\cup D$. 
Hence, the condition (U5) holds for $\{(U_\alpha,\varphi_\alpha)\}_{\alpha\in A}$.

For any $i\in \mathcal{I}$, $\varphi_i^{-1}(\Cl{B(\nu_i/2)})$ and $U_{i}$ are contained in $\Int(K_i)\cap f^{-1}(y_i-\nu_i/2,y_i+\nu_i/2)\setminus K_{i-1}$. 
In particular $U_{j}\cap \varphi_i^{-1}(\Cl{B(\nu_i/2)})=\emptyset$ for $j\neq i$ (note that $(y_i-2\nu_i,y_i+2\nu_i)\cap (y_j-2\nu_j,y_j+2\nu_j)$ is empty). 
The set $U_\beta$ is away from $\varphi_i^{-1}(\Cl{B(\nu_i/2)})$ for any $j\in \mathcal{I}$ and $\beta \in B_j$ since $U_\beta$ is contained in $V_j\cap f^{-1}(y_j-\nu_j,y_j+\nu_j)$ and $V_j\cap K_j=\emptyset$.
For $x \in C_j$, $U_x$ is contained in $f^{-1}(y_j-\nu_j,y_j+\nu_j)\setminus \varphi_j^{-1}(\Cl{B(\nu_j/2)})$, and thus $U_x \cap \varphi_i^{-1}(\Cl{B(\nu_i/2)})=\emptyset$. 
The intersection $U_x \cap \varphi_i^{-1}(\Cl{B(\nu_i/2)})$ is empty for any $x\in N\setminus \left(f^{-1}(\Cl{\Delta})\cup M\right)$ since $U_x$ is away from $f^{-1}([y_j-\nu_j/2,y_j+\nu_j/2])$ for any $j\in \mathcal{I}$. 
Lastly, for $x\in D$, we take $c\in \N$ and $j_1,\ldots, j_e\in \mathcal{I}$ as above (i.e.~$x\in K_c\setminus K_{c-1}$ and $j_1,\ldots, j_e\leq c+1$). 
By the construction, $U_x$ is contained in $\Int(K_{c+1})$ and away from $f^{-1}([y_{j_m}-\nu_{j_m}/2,y_{j_m}+\nu_{j_m}/2])$ for $m=1,\ldots, e$. 
In particular $U_x$ does not intersect with $\varphi_i^{-1}(\Cl{B(\nu_i/2)})$. 
We can thus conclude that (U6) holds for $\{(U_\alpha,\varphi_\alpha)\}_{\alpha\in A}$, completing the proof of Lemma~\ref{T:existence system cdnt nbhd}.
\end{proof}

Let $\{(U_\alpha,\varphi_\alpha)\}_{\alpha\in A}$ be a system of coordinate neighborhoods of $N$ satisfying the conditions (U1)--(U6) in Lemma~\ref{T:existence system cdnt nbhd}. 
We put $U = \cup_{i\in \mathcal{I}} U_{i}$ and $\mathcal{V}_0 = C^\infty(N)$. 
Since $N$ is paracompact, there exists a locally finite covering $\{L_\alpha\}_{\alpha\in A}$ consisting of closed subsets $L_\alpha\subset U_\alpha$ (\cite[Lemma 5.1.6]{Engelking}).
Note that $L_\alpha$ is compact by the condition (U1). 
We define a subspace $\mathcal{V}_k$~($k=1,2,3,4$) of $C^\infty(N)$ as follows:
{\allowdisplaybreaks
\begin{align*}
\mathcal{V}_1 = &\{g\in C^\infty(N)~|~ g|_{N-U}=f|_{N-U}\}, \\
\mathcal{V}_2 = &\{g\in C^\infty(N)~|~\Delta(g) =\Delta\}, \\
\mathcal{V}_3 = &\{g\in C^\infty(N)~|~ \Delta(g) = \Delta,\hspace{.3em} \Sigma(g) = \Sigma\},\\
\mathcal{V}_4 = &\{f\}.  
\end{align*}
}
\noindent
We will prove (1) of Theorem~\ref{T:suff condi stability} by showing the following Claim~$k$~($k=1,2,3,4$):

\V{.8em}

\noindent
{\bf Claim~$\boldsymbol{k}$.}
There exist an open neighborhood $\mathcal{U}_{k}\subset \mathcal{V}_{k-1}$ of $f$ and a continuous mapping $\theta_k:\mathcal{U}_k\to \mathcal{V}_k$ such that $\theta_k(f)=f$ and $g$ is $\mathcal{A}$-equivalent to $\theta_k(g)$ for any $g\in \mathcal{U}_k$. 

\V{.8em}

\noindent
Note that the statement (1) of Theorem~\ref{T:suff condi stability} immediately follows from Claims~$1,\ldots,4$ (any mapping $g\in \mathcal{U}=(\theta_3\circ \theta_2\circ \theta_1)^{-1}(\mathcal{U}_4)$ is $\mathcal{A}$-equivalent to $f$). 

\subsection*{Proof of Claim~1}
We will use the following lemma:

\begin{lemma}[{\cite[Theorem 3.6.1]{duPlessisWall}}]\label{T:strongstabilitysubmersion}
Let $h:N\to P$ be $C^r$-mapping ($1\leq r\leq \infty$) and $U\subset N$ be an open neighborhood of $\Sigma(h)$. 
There exists an open neighborhood $\mathcal{U}\subset C^\infty(N,P)$ of $h$ with respect to the $C^1$-topology and a mapping $\beta:\mathcal{U}\to \Diff(N)$ which is continuous with respect to the $C^s$-topologies for any $s\in \{1,\ldots, r\}$ which satisfy the following conditions: 

\begin{itemize}
	
\item 
$\beta(h)$ is equal to the identity $\id_N$.

\item 
$g\circ \beta(g)|_{N-U} $ is equal to $h|_{N-U}$ for any $g\in \mathcal{U}$.
	
\item
$\beta(g)(x)=x$ for any $x\in N-U$ with $h(x)=g(x)$. 	
	
\end{itemize}
\end{lemma}

\noindent
By applying this lemma to $f$ and the open set $U$, we can take an open neighborhood $\mathcal{U}_1\subset C^\infty(N)$ of $f$ and a mapping $\beta:\mathcal{U}_1\to \Diff(N)$ so that: 

\begin{itemize}

\item 
$\beta$ is continuous with respect to the $C^s$-topologies for any $s\in \N\cup \{\infty\}$,

\item 
$\beta(f)$ is equal to $\id_N$, and

\item 
$g\circ \beta(g)|_{N-U} = f|_{N-U}$ for any $g\in \mathcal{U}_1$. 

\end{itemize}

\noindent
We define a mapping $\theta_1:\mathcal{U}_1\to \mathcal{V}_1$ by $\theta_1(g)=g\circ \beta(g)$. 
Since any diffeomorphism is proper, $\theta_1$ is continuous (see \cite[\S.2, Proposition 1]{MatherII}). 
The mapping $\theta_1$ satisfies the desired conditions. 
\subsection*{Proof of Claim~2}

Let $\gamma_i = \nu_i^{1/\nu_i}$, where $\nu_i$ appears in the definition of $\Phi_i$. 
We can easily check that $\gamma_i$ is less than $\nu_i$. 
Using $\gamma_i$, we next define $e_\alpha>0$ as follows: 
\[
e_\alpha = \begin{cases}
\dfrac{\gamma_i}{4n} & (\alpha = i\in \mathcal{I})\\[5pt]
1 & (\alpha \not\in \mathcal{I}).
\end{cases}
\]
Let $e = \{e_\alpha\}_{\alpha\in A}$ and $\D \mathcal{U}_2 = \mathcal{V}_1 \cap N_2(f,L,\varphi,\id_\R,e)$, which is an open neighborhood of $f$ in $\mathcal{V}_1$. 
For the sake of simplicity, throughout the paper, we denote the norms $\norm{D^k(h\circ \varphi_\alpha^{-1})(\varphi_\alpha(x))}$ and $\norm{h\circ \varphi_\alpha^{-1}}_{s,\varphi_\alpha(L_\alpha)}$ by $\norm{D^kh(x)}$ and $\norm{h}_{s,L_\alpha}$, respectively, for $s\geq 0$, $x\in L_\alpha$ and $h\in C^\infty(N)$.  
Since $f\circ\varphi_{i}^{-1}(w_1,\ldots,w_n) = \pm w_1^2\pm \cdots \pm w_n^2 + y_i$, we can directly calculate $\norm{D^1f(x)}$ for $x\in L_i$ as follows: 
\[
\norm{D^1f(x)} = \sqrt{\sum_{m=1}^{n}\left(\frac{\Pa f\circ \varphi_{i}^{-1}}{\Pa w_m}(\varphi_{i}(x))\right)^2} = 2 \norm{\varphi_{i}(x)}. 
\]
Hence the norm $\norm{D^1g(x)}$ for $x\in L_{i}$ is estimated as follows: 
\[
\norm{D^1g(x)}\geq \norm{D^1f(x)}- \norm{g- f}_{2,L_i} > 2 \norm{\varphi_{i}(x)} - \dfrac{\gamma_i}{4}.
\]
We can deduce from this inequality that all the critical points of $g$ are contained in $\cup_{i\in \mathcal{I}} \varphi_i^{-1}(\Cl{B(\gamma_i/8)})$ (note that $g=f$ in $N\setminus U$, in particular $g$ is regular in $N\setminus U$). 
By the condition (U6), $\varphi_i^{-1}(\Cl{B(\nu_i/2)})$ is contained in $L_i$. 
Since $\norm{g-f}_{2,L_{i}}$ is less than $\gamma_i/4n < \nu_i/2n$, we can deduce from Lemma~\ref{T:crit pt unique} that there exists exactly one critical point of $g$ in $\varphi_i^{-1}(\Cl{B(\nu_i/2)})$ for each $i\in \mathcal{I}$. 
We denote this critical point of $g$ by $x_{i,g}$, and let $y_{i,g} = g(x_{i,g})$. 
Note that $x_{i,g}$ is in $\varphi_i^{-1}(\Cl{B(\gamma_i/8)})$. 
The norm $|y_i-y_{i,g}|$ can be estimated as follows: 
{\allowdisplaybreaks
\begin{align*}
|y_i-y_{i,g}| & = |f(x_i)-g(x_{i,g})|\\
& \leq |f(x_i)-f(x_{i,g})| + |f(x_{i,g})-g(x_{i,g})| \\
& < \left(\max_{x\in \varphi_i^{-1}(\Cl{B(\gamma_i/8))}}\norm{df_x} \right) \cdot |x_i-x_{i,g}| + \frac{\gamma_i}{4} < \frac{\gamma_i}{2}. 
\end{align*}
}

We first construct a diffeomorphism $\psi_g:\R\to \R$ such that $\Delta(\psi_g^{-1}\circ g)$ is equal to $\Delta$. 
Although the construction below is same as that for the mapping $h$ in \cite{Dimca}, we will briefly explain the construction for completeness of this manuscript. 
We take a smooth function $\rho:\R\to \R$ so that: 

\begin{itemize}

\item 
$\rho(t)= \rho(-t)$, 

\item $\rho|_{[0,\infty)}$ is monotone decreasing,

\item 
$\rho(y)=1$ if $|y|\leq 1$, and $\rho(y)=0$ if $|y|\geq 2$, and 

\item
$\lvert\rho'(y)\rvert<2$ for any $y\in \R$. 

\end{itemize}

\noindent
Using the function $\rho$, we define a function $h_i:\R \to \R$ as follows: 
\[
h_i(y) = \begin{cases}
y + (y_{i,g}-y_i)\rho\left(\dfrac{4(y- y_i)}{\nu_i}\right) & (y\in (y_i-\nu_i, y_i+\nu_i)) \\
y & (y\not\in (y_i-\nu_i, y_i+\nu_i)). 
\end{cases}
\]
Note that $h_i$ is a diffeomorphism as $h_i$ is surjective and 
\[
h_i'(y)> 1-4\gamma_i/\nu_i >1-4\nu_i^{1/\nu_i-1}>1-4\nu_i>0~(\because \nu_i < 1/4).
\] 
We further define a mapping $\psi_g:\R\to \R$ as follows: 
\[
\psi_g(y) = \begin{cases}
h_i(y) & (y\in (y_i-\nu_i,y_i+\nu_i)\mbox{ for some }i\in \mathcal{I}) \\
y & (y \not\in \bigcup_{i\in \mathcal{I}} (y_i-\nu_i,y_i+\nu_i)). 
\end{cases}
\]
One can prove that $\psi_g$ is a diffeomorphism in the same way as the proof that the mapping $h$ in \cite{Dimca} is a diffeomorphism. (Note that we did not merely put $\gamma_i = \nu_i/C$ (for $C\gg 1$) but $\gamma_i=\nu_i^{1/\nu_i}$ in order to guarantee that $\psi_g$ is a diffeomorphism.
See the proof in \cite{Dimca} for details.) 
One can also easily verify that $\Delta(\psi_g^{-1}\circ g)$ is equal to $\Delta$. 

For each $i\in \mathcal{I}$, we take a function $\eta_i:N\to[0,1]$ so that $\eta_i|_{K_{d(i)}} \equiv 0$ and $\eta_i|_{N\setminus K_{d(i)+1}} \equiv 1$. 
Using $\eta_i$ we define a mapping
\[
\widetilde{\eta}_{i,g}:(f^{-1}(y_i)\cap V_i)\times (y_i-2\nu_i,y_i+2\nu_i) \to (f^{-1}(y_i)\cap V_i)\times (y_i-2\nu_i,y_i+2\nu_i)
\]
as follows: 
\[
\widetilde{\eta}_{i,g}(x,y) = \left(x, y + (y_{i,g}-y_i)\rho\left(\frac{4(y-y_i)}{\nu_i}\right)\eta_i(\Phi_i(x,y))\right). 
\]
It is easy to see that $\widetilde{\eta}_{i,g}$ is a diffeomorphism and it is the identity mapping on the complement of $(f^{-1}(y_i)\cap V_i)\times (y_i-\nu_i/2,y_i+\nu_i/2)$. 
In particular we can define $\Psi_g^1:N\to N$ as follows (note that $(y_i-2\nu_i,y_i+2\nu_i)\cap (y_j-2\nu_j,y_j+2\nu_j)=\emptyset$ for $i\neq j$): 
\[
\Psi_g^1(x) = \begin{cases}
\Phi_i \circ \widetilde{\eta}_{i,g} \circ \Phi_i^{-1} (x) & \left(x\in f^{-1}(y_i-2\nu_i,y_i+2\nu_i)\cap V_i\mbox{ for some }i\in \mathcal{I}\right)\\
x & \left(x\not\in \bigcup_{i\in \mathcal{I}} (f^{-1}(y_i-\nu_i,y_i+\nu_i)\cap V_i)\right).
\end{cases}
\]
Note that if $f$ is quasi-proper, $f^{-1}(y_i-2\nu_i,y_i+2\nu_i)\cap V_i$ is empty for each $i\in \mathcal{I}$, and thus $\Psi_g^1$ is the identity mapping.  

\begin{lemma}

The mapping $\Psi_g^1$ is a diffeomorphism.

\end{lemma}

\begin{proof}
We can easily verify that $\Psi_g^1$ is a bijection. 
Let $Q\subset N$ be the boundary of $\cup_{i\in \mathcal{I}} (f^{-1}(y_i-\nu_i,y_i+\nu_i)\cap V_i)$ (not as a manifold, but in the sense of general topology).  
We can immediately deduce from the definition that $\Psi_g^1$ is locally diffeomorphic at any point in $N\setminus Q$.  
In what follows, we will show that for any $x\in Q$, there exists a neighborhood of $x$ on which $\Psi_g^1$ is the identity mapping (especially locally diffeomorphic). 

Let $\{x_m\}_{m\in \N}$ be a sequence of points in $\cup_{i\in \mathcal{I}} (f^{-1}(y_i-\nu_i,y_i+\nu_i)\cap V_i)$ converging to $x$. 
Since $\{x_m\}_{m\in \N}$ is a converging sequence, it is contained in $K_l$ for some $l\in \N$. 
As $K_i\cap V_i=\emptyset$ for any $i\in \mathcal{I}$, there exist only finitely many $i\in \mathcal{I}$ such that $f^{-1}(y_i-\nu_i,y_i+\nu_i)\cap V_i \cap K_l$ is not empty. 
In particular, by taking a subsequence of $\{x_m\}_{m\in \N}$ if necessary, we can assume that $\{x_m\}_{m\in \N}$ is contained in $f^{-1}(y_j-\nu_j,y_j+\nu_j)\cap V_j$, and thus $x\in f^{-1}([y_j-\nu_j,y_j+\nu_j])$ for some $j\in \mathcal{I}$.  
If $f(x)=y_j\pm \nu_j$, $x$ is contained in either of the preimages $f^{-1}(y_j+ \nu_j/2,y_j+ 2\nu_j)$ or $f^{-1}(y_j-2 \nu_j,y_j- \nu_j/2)$, and $\Psi_g^1$ is the identity mapping on it. 
Suppose that $x$ is in $f^{-1}(y_j-\nu_j,y_j+\nu_j)$. 
As $f^{-1}(y_j-\nu_j,y_j+\nu_j)\cap V_j$ is open and $x$ is a point in the boundary (especially the complement) of it, $x$ is contained in $N\setminus V_j \subset \Int(K_{d(j)})$.
Furthermore, the composition $\Phi_j\circ \widetilde{\eta}_{j,g}\circ \Phi_j^{-1}$ is the identity mapping on the intersection $f^{-1}(y_j-\nu_j,y_j+\nu_j)\cap V_j\cap K_{d(j)}$.
Thus, $\Psi_g^1$ is the identity mapping on $f^{-1}(y_j-\nu_j,y_j+\nu_j)\cap \Int(K_{d(j)})$, which is a neighborhood of $x$. 
\end{proof}

Since composing a diffeomorphism of the source does not affect the critical value set, $\psi_g^{-1}\circ g \circ \Psi_g^1$ is contained in $\mathcal{V}_2$ for $g\in \mathcal{U}_2$.

\begin{lemma}\label{T:continuity theta2}
The mapping $\theta_2:\mathcal{U}_2 \to \mathcal{V}_2$ defined as $\theta_2(g) = \psi_g^{-1}\circ g \circ \Psi_g^1$ is continuous with respect to the $C^s$-topologies for any $s\geq 2$, and thus with respect to the $C^\infty$-topologies. 

\end{lemma}

\begin{remark}\label{Re:case quasi-proper}
Since $f$ is not necessarily proper, the mapping $g\mapsto \psi_g^{-1}\circ g$ is not continuous in general. 
We can also check that the mapping $g\mapsto\Psi_g^1$ is not continuous in general, either. 
However, as we will show, the mapping (denoted by $\theta_2''$ in the proof below) $g\mapsto \psi_g$ is continuous. 
Furthermore, if $f$ is quasi-proper, the mapping $g\mapsto \psi_g^{-1}\circ g$ is continuous since $\Psi_g^1$ is the identity mapping. 

\end{remark}

\begin{proof}[Proof of Lemma~\ref{T:continuity theta2}]
Throughout the proof, we will assume that any mapping space is endowed with the $C^s$-topology for $s\geq 2$. 
Let $g\in \mathcal{U}_2\subset \mathcal{V}_1$. 
By the definition of $\mathcal{V}_1$, $g$ is equal to $f$ on $N\setminus U$ (where $U=\cup_{i\in \mathcal{I}}U_i$, as defined above). 
Since $U_i$ is contained in $f^{-1}(y_i-\nu_i/2,y_i+\nu_i/2)$ and $\Psi_g^1$ (resp.~$\psi_g$) is the identity mapping on the complement of $\cup_{i\in \mathcal{I}} f^{-1}(y_i-\nu_i/2,y_i+\nu_i/2)$ (resp.~$\cup_{i\in \mathcal{I}} (y_i-\nu_i/2,y_i+\nu_i/2)$), the composition $\psi_g^{-1}\circ g \circ \Psi_g^1$ is equal to $g$($=f$) on $N\setminus \cup_{i\in \mathcal{I}} f^{-1}(y_i-\nu_i/2,y_i+\nu_i/2)$. 
For each $i\in \mathcal{I}$, the diffeomorphism $\widetilde{\eta}_{i,g}$ (appearing in the definition of $\Psi_g^1$) is equal to $p_1\times (\psi_g\circ p_2)$ on $f^{-1}(y_i-\nu_i/2,y_i+\nu_i/2)\setminus K_{d(i)+1}$, where $p_m$ is the projection to the $m$-th component.
As $U_i\subset K_{d(i)+1}$ for each $i\in \mathcal{I}$, we can conclude that the composition $\psi_g^{-1}\circ g \circ \Psi_g^1$ is equal to $g$($=f$) on the complement $N\setminus \cup_{i\in \mathcal{I}} \tilde{N}_i$, where $\tilde{N}_i$ is defined as follows: 
\[
\tilde{N}_i = f^{-1}([y_i-\nu_i/2,y_i+\nu_i/2])\cap K_{d(i)+1}.
\]
Note that $\tilde{N}_i$ is equal to $N_i \cup (K_{d(i)}\cap f^{-1}([y_i-\nu_i/2,y_i+\nu_i/2]))$. 
In what follows, we will prove the following two claims: 

\begin{enumerate}

\item 
The mapping $\theta_2':\mathcal{U}_2\to C^\infty(\tilde{N}_i,N)$ defined as $\theta_2'(g)=\Psi_g^1|_{\tilde{N}_i}$ is continuous. 

\item 
The mapping $\theta_2'':\mathcal{U}_2 \to C^\infty(\R)$ defined as $\theta_2''(g)=\psi_g$ is continuous.

\end{enumerate}

We first prove continuity of $\theta_2'$.
By the definition, $\Psi_g^1$ is the identity mapping on $K_{d(i)}\cap f^{-1}([y_i-\nu_i/2,y_i+\nu_i/2])$ and on the complement of $\cup_{i\in \mathcal{I}}f^{-1}(y_i-\nu_i/2,y_i+\nu_i/2)$. 
Moreover, one can easily verify that $\Psi_g^1$ preserves the image $\Phi_i(W\times (y_i-\nu_i,y_i+\nu_i))$ for any $i\in \mathcal{I}$ and $W\subset f^{-1}(y_i)\cap V_i$. 
We can thus deduce from the conditions (U3) and (U4) that $\Psi_g^1(U_\alpha)=U_\alpha$ for any $\alpha \in A$. 
By Theorem~\ref{T:basis Whitney topology}, it is enough to show that for any $g\in \mathcal{U}_2$ and system $\varepsilon =\{\varepsilon_\alpha\}_{\alpha \in A}$ of positive numbers, there exists a system $\delta =\{\delta_\alpha\}_{\alpha\in A}$ satisfying the following:
\[
\theta_2'(N_s(g,L,\varphi,\id,\delta)\cap \mathcal{U}_2)\subset N_s(\Psi_g^1,L\cap\tilde{N}_i,\varphi,\varphi,\varepsilon). 
\]

By the condition (U5), there exist only finitely many $\alpha\in A$ satisfying the condition $L_\alpha \cap N_i\neq\emptyset$, in particular we can take $\kappa_i = \min \left\{\varepsilon_\alpha ~|~  L_\alpha \cap N_i\neq\emptyset\right\}$. 
Suppose that $\alpha$ satisfies (U3-2), that is, $U_\alpha = \Phi_i(W_\alpha \times (y_i-\nu_i,y_i+\nu_i))$ and $\varphi_\alpha = (\phi_\alpha\times \id)\circ \Phi_i^{-1}$. 
The following equality holds for $h\in \mathcal{U}_2$ and $(x,y)\in \varphi_\alpha(L_\alpha) \subset \phi_\alpha(W_\alpha) \times (y_i-\nu_i,y_i+\nu_i)$: 
{\allowdisplaybreaks
\begin{align*}
& \varphi_\alpha\circ \Psi_h^1|_{\tilde{N}_i}\circ \varphi_\alpha^{-1}(x,y)-\varphi_\alpha\circ \Psi_g^1|_{\tilde{N}_i}\circ \varphi_\alpha^{-1}(x,y)\\
& (\phi_\alpha \times \id)\circ \widetilde{\eta}_{i,g}\circ (\phi_\alpha^{-1} \times \id)(x,y)-(\phi_\alpha \times \id)\circ \widetilde{\eta}_{i,h}\circ (\phi_\alpha^{-1} \times \id)(x,y)\\
= &\left(0, (y_{i,g}-y_{i,h})\rho\left(\frac{4(y-y_i)}{\nu_i}\right)\eta_i(\Phi_i(\phi_\alpha^{-1}(x),y))\right)\\
= &(y_{i,g}-y_{i,h})\left(0, \rho\left(\frac{4(y-y_i)}{\nu_i}\right)\eta_i(\Phi_i(\phi_\alpha^{-1}(x),y))\right).
\end{align*}
}%
We define $m_{i,s}\in \R$ as follows: 
\[
m_{i,s} = \max \left\{\left.\norm{\left(0, \rho\left(\frac{4(y-y_i)}{\nu_i}\right)\eta_i(\Phi_i(\phi_\alpha^{-1}(x),y))\right)}_{s,\varphi_\alpha(L_\alpha)}\right|\begin{minipage}[c]{26mm}
$\alpha$ satisfies (U3-1)

and $L_\alpha \cap N_i \neq \emptyset$
\end{minipage}\right\}.
\]
From the calculation above, we obtain the following estimate for any $h\in \mathcal{U}_2$ and $\alpha \in A$ with $L_\alpha \cap N_i\neq \emptyset$: 
\[
\norm{\varphi_\alpha\circ \Psi_h^1|_{\tilde{N}_i}\circ \varphi_\alpha^{-1}-\varphi_\alpha\circ \Psi_g^1|_{\tilde{N}_i}\circ \varphi_\alpha^{-1}}_{s,\varphi_\alpha(L_\alpha)}\leq \lvert y_{i,g}-y_{i,h}\rvert \cdot m_{i,s}
\]
We take a system $\{\delta_\alpha\}_{\alpha \in A}$ of positive numbers as follows: 
\[
\delta_\alpha = 
\begin{cases}
\dfrac{\kappa_i}{m_{i,s} \left(\sqrt{n}\norm{g}_{1,L_i} + 1\right)} & (\alpha = i\in \mathcal{I})\\
1 & (\alpha\not\in \mathcal{I}). 
\end{cases}
\]
Let $h \in N_s(g,L,\varphi,\id,\delta)\cap \mathcal{U}_2$. 
Since $g$ and $h$ are functions in $\mathcal{U}_2$, the norms $\norm{g-f}_{2,L_i}$ and $\norm{h-f}_{2,L_i}$ are less than $e_i < \nu_i/n$. 
As the critical points $x_{i,g}$ and $x_{i,h}$ are contained in $\varphi_i^{-1}(\Cl{B(\nu_i/2)})$, and this set is a subset of $L_i$, we can obtain the following estimate from Lemma~\ref{T:evaluation critical point/value}:
{\allowdisplaybreaks
\begin{align*}
\left|y_{i,h}-y_{i,g}\right|<& \left(\sqrt{n}\norm{g}_{1,\varphi_i^{-1}(\Cl{B(\nu_i/2)})}+1\right)\norm{g-h}_{1,\varphi_i^{-1}(\Cl{B(\nu_i/2)})}\\
\leq &\left(\sqrt{n}\norm{g}_{1,L_i}+1\right)\norm{g-h}_{1,L_i} \\
<& \left(\sqrt{n}\norm{g}_{1,L_i}+1\right)\delta_i < \frac{\kappa_i}{m_{i,s}}. 
\end{align*}
}%
We thus obtain the following estimate for $\alpha \in A$ with $L_\alpha \cap N_i\neq \emptyset$: 
\[
\norm{\varphi_\alpha\circ \Psi_h^1|_{\tilde{N}_i}\circ \varphi_\alpha^{-1}-\varphi_\alpha\circ \Psi_g^1|_{\tilde{N}_i}\circ \varphi_\alpha^{-1}}_{s,\varphi_\alpha(L_\alpha)} <  \kappa_i< \varepsilon_\alpha. 
\]
On the other hand, $\Psi_h^1$ is the identity mapping on $\tilde{N}_i\cap K_{d(i)}$ for any $h\in \mathcal{U}_2$. 
Hence, $\Psi_h^1 = \Psi_g^1$ on $\tilde{N}_i\setminus N_i$, in particular $\norm{\varphi_\alpha\circ \Psi_h^1|_{\tilde{N}_i}\circ \varphi_\alpha^{-1}-\varphi_\alpha\circ \Psi_g^1|_{\tilde{N}_i}\circ \varphi_\alpha^{-1}}_{s,\varphi_\alpha(L_\alpha)}$ is equal to $0$, especially less than $\varepsilon_\alpha$ for $\alpha \in A$ with $L_\alpha \cap N_i=\emptyset$. 
Therefore, the diffeomorphism $\Psi_h^1|_{\tilde{N}_i}$ is contained in $N_s(\Psi_g^1,L\cap \tilde{N}_i,\varphi,\varphi,\varepsilon)$. 

We next prove continuity of $\theta_2''$. 
Let $L'=\{[r_\beta,s_\beta]\}_{\beta\in B}$ be a system of closed intervals such that for each $i\in \mathcal{I}$, there exist only finitely many $\beta\in B$ with $(y_i-\nu_i,y_i+\nu_i)\cap [r_\beta,s_\beta] \neq \emptyset$. 
Again, by Theorem~\ref{T:basis Whitney topology}, it is enough to show that for any $g\in \mathcal{U}_2$ and system $\varepsilon = \{\varepsilon_\beta\}_{\beta \in B}$ of positive numbers, there exists a system $\delta = \{\delta_\alpha\}_{\alpha\in A}$ such that:
\[
\theta_2''(N_s(g,L,\varphi,\id,\delta)\cap \mathcal{U}_2) \subset N_s(\psi_g,L',\id,\id,\varepsilon). 
\]
We take a positive number $\kappa_i'$ as follows: 
\[
\kappa_i' = \min\left\{\varepsilon_\beta>0 ~\left|\right. (y_i-\nu_i,y_i+\nu_i)\cap [r_\beta,s_\beta] \neq \emptyset\right\}. 
\]
For $h\in \mathcal{U}_2$, the norm $\norm{\psi_h-\psi_g}_{s,[r_\beta,s_\beta]}$ is equal to $0$ if $[r_\beta,s_\beta]$ is away from the union $\cup_{i\in \mathcal{I}} (y_i-\nu_i,y_i+\nu_i)$ since $\psi_h$ is the identity mapping outside the union.
On the other hand, the following holds for $y\in (y_i-\nu_i,y_i+\nu_i)$: 
\[
\psi_g(y)-\psi_h(y) = (y_{i,g}-y_{i,h}) \rho\left(\frac{4(y-y_i)}{\nu_i}\right). 
\]
Let $m_{i,s}' = \norm{\rho\left(\frac{4(y-y_i)}{\nu_i}\right)}_{s,(y_i-\nu_i,y_i+\nu_i)}$ and we take a system $\{\delta_\alpha\}_{\alpha\in A}$ of positive numbers as follows: 
\[
\delta_\alpha = 
\begin{cases}
\dfrac{\kappa_i'}{m_{i,s}'\left(\sqrt{n}\norm{g}_{1,L_i} + 1\right)} & (\alpha = i\in \mathcal{I}) \\
1 & (\alpha\not\in \mathcal{I}). 
\end{cases}
\]
Let $h\in N_s(g,L,\varphi,\id,\delta)\cap \mathcal{U}_2$. 
As in the previous paragraph, one can deduce from Lemma~\ref{T:evaluation critical point/value} that $\lvert y_{i,g}-y_{i,h}\rvert$ is less than $\kappa_i/m_{i,s}'$.
Thus, the norm $\norm{\psi_h-\psi_g}_{s,[r_\beta,s_\beta]}$ can be estimated as follows: 
{\allowdisplaybreaks
\begin{align*}
\norm{\psi_h-\psi_g}_{s,[r_\beta,s_\beta]} \leq &\sup_{\substack{i\in \mathcal{I} \\ (y_i-\nu_i,y_i+\nu_i)\cap [r_\beta,s_\beta] \neq \emptyset}}\left|y_{i,h} - y_{i,g}\right| m_{i,s}' \\
< &\sup_{\substack{i\in \mathcal{I} \\ (y_i-\nu_i,y_i+\nu_i)\cap [r_\beta,s_\beta] \neq \emptyset}} \kappa_i'.
\end{align*}
}%
Since $\kappa_i'\leq \varepsilon_\beta$ for any $i\in \mathcal{I}$ with $(y_i-\nu_i,y_i+\nu_i)\cap [r_\beta,s_\beta] \neq \emptyset$, the norm $\norm{\psi_h-\psi_g}_{s,[r_\beta,s_\beta]}$ is less than $\varepsilon_\beta$. 
Therefore, the image $\theta_2''(N_s(g,L,\varphi,\id,\delta)\cap \mathcal{U}_2)$ is contained in $N_s(\psi_g,L',\id,\id,\varepsilon)$. 

Since the mappings $\theta_2'$ and $\theta_2''$ are continuous and $\tilde{N}_i$ is compact, the following mapping is also continuous (see \cite[\S.2, Propositions 1, 2 and 5]{MatherII}): 
\[
\tau_i:\mathcal{U}_2\to C^\infty(\tilde{N}_i),\hspace{.5em} g \mapsto \psi_g^{-1}\circ g\circ \Psi_g^1|_{\tilde{N}_i}. 
\]
Thus, for any $g\in \mathcal{U}_2$ and any system $\{\varepsilon_\alpha\}_{\alpha\in A}$ of positive numbers, we can take a system $\delta^i=\{\delta_\alpha^i\}_{\alpha\in A}$ so that $\tau_i(N_s(g,L,\varphi,\id,\delta^i)\cap \mathcal{U}_2)$ is contained in $N_s(\tau_i(g),L\cap \tilde{N}_i,\varphi,\id,\varepsilon/2)$. 
By the definitions of $\psi_h$ and $\Psi_h^1$, the function $\tau_i(h)$ depends only on $y_{i,h}$. 
In particular, $\tau_i(N_s(g,L,\varphi,\id,\gamma)\cap \mathcal{U}_2) $ is in $ N_s(\tau_i(g),L\cap \tilde{N}_i,\varphi,\id,\varepsilon)$ for any system $\{\gamma_\alpha\}_{\alpha\in A}$ with $\gamma_{i}=\delta^i_{i}$ ($i\in \mathcal{I}$). 
We take a system $\{\tilde{\delta}_\alpha\}_\alpha$ as follows: 
\[
\tilde{\delta}_\alpha = 
\begin{cases}
\delta^i_{i} & (\alpha = i\in \mathcal{I}) \\
1 & (\alpha\not\in \mathcal{I}). 
\end{cases}
\]
Let $h \in N_s(g,L,\varphi,\id,\tilde{\delta})\cap \mathcal{U}_2$.
The intersection $L_\alpha \cap \supp(\theta_2(g)-\theta_2(h))$ is contained in $\cup_{i\in \mathcal{I}}(\tilde{N}_i \cap L_\alpha)$, in particular the norm $\norm{\theta_2(h)- \theta_2(g)}_{s,L_\alpha}$ can be estimated as follows:
{\allowdisplaybreaks
\begin{align*}
\norm{\theta_2(h)- \theta_2(g)}_{s,L_\alpha} 
=& \norm{\theta_2(h)- \theta_2(g)}_{s,\cup_{i\in \mathcal{I}} (\tilde{N}_i\cap L_\alpha)}\\
\leq & \sup_{i\in \mathcal{I}}\norm{\theta_2(h)- \theta_2(g)}_{s,\tilde{N}_i\cap L_\alpha}\\
=& \sup_{i\in \mathcal{I}}\norm{\tau_i(h)- \tau_i(g)}_{s,\tilde{N}_i\cap L_\alpha} \leq \varepsilon_\alpha/2 <\varepsilon_\alpha. 
\end{align*}
}%
Thus, $\theta_2(N_s(g,L,\varphi,\id,\tilde{\delta}))$ is contained in $N_s(\theta_2(g),L,\varphi,\id,\varepsilon)$ and by Theorem~\ref{T:basis Whitney topology}, $\theta_2$ is continuous. 
\end{proof}

The mapping $\theta_2:\mathcal{U}_2\to \mathcal{V}_2$ satisfies the desired conditions. 
Indeed, it is obvious by the definition that $\theta_2(g)$ is $\mathcal{A}$-equivalent to $g$ for any $g\in \mathcal{U}_2$. 
As $y_{i,f} = y_i$ for each $i\in \mathcal{I}$, $\Psi_f^1$ and $\psi_f$ are both the identity mappings, and thus $\theta_2(f) = f$.  

\subsection*{Proof of Claim~3}

We will use the following lemma in the proof:

\begin{lemma}[{\cite[\S.7, Lemma 2]{MatherII}}]\label{T:continuity integral vectfield}
Let $U$ be a manifold without boundaries and $\pi_U:U\times I\to U$ be the projection, where $I=[0,1]\subset \R$ is the unit interval. 
There exists an open neighborhood $\mathcal{O}_U \subset \Gamma(\pi_U^\ast TU)$ of the zero-section (with respect to the topology $\tau W^0$) such that we can take a mapping 
\[
\theta:\mathcal{O}_U\to C^\infty(U\times I ,U)
\]
by taking a flow of a vector field in $\mathcal{O}_U$ and $\theta$ is continuous with respect to the topologies $\tau W^s$ (for $0\leq s \leq \infty$). 

\end{lemma}

Let $e = \{e_\alpha\}_{\alpha\in A}$ be the system we took in the proof of Claim 2 and $\D \mathcal{U}_3' = \mathcal{V}_2 \cap N_2(f,L,\varphi,\id,e)$, which is an open neighborhood of $f$ in $\mathcal{V}_2$. 
As we observed, for any $g\in \mathcal{U}_3'$ and $i\in \mathcal{I}$ there exists exactly one critical point $x_{i,g}$ of $g$ in $\Cl{\varphi_i^{-1}(B(\gamma_i/8))}$. 
We define a vector field $X_g$ on $N$ as follows: 
\[
X_g(w) = \begin{cases}
\lambda(w)(\varphi_i(x_{i,g}) - \varphi_i(x_i)) & (w\in \Cl{\varphi_i^{-1}(B(\gamma_i/4))}\mbox{ for some }i\in \mathcal{I})\\
0 & (\mbox{otherwise}),  
\end{cases}
\]
where $\lambda:N\to [0,1]$ is a non-negative valued function satisfying the conditions $\lambda|_{\cup_{i\in\mathcal{I}}\Cl{\varphi_i^{-1}(B(\gamma_i/8))}}\equiv 1$ and $\lambda|_{N\setminus\cup_{i\in\mathcal{I}}\Cl{\varphi_i^{-1}(B(\gamma_i/4))}}\equiv 0$. (Note that we identify a vector field on  $U_{i}$($\supset\Cl{\varphi_i^{-1}(B(\gamma_i/4))}$) with an element in $C^\infty(U_{i})^n$ via the chart $\varphi_i$.)
We define a mapping $\chi:\mathcal{U}_3'\to \Gamma(\pi_N^\ast TN)$ by $\chi(g) =X_g\circ \pi_N$. 
By Lemma~\ref{T:continuity integral vectfield} we can take an open neighborhood  $\mathcal{O}_N\subset \Gamma(\pi_N^\ast TN)$ of the zero-section so that the mapping $\theta: \mathcal{O}_N\to C^\infty(N\times I ,N)$ defined by taking a flow is continuous. 
Let $\mathcal{U}_3 = \chi^{-1}(\mathcal{O}_N)$ and we define a diffeomorphism $\Psi_g^2:N\to N$ for $g\in \mathcal{U}_3$ as follows:
\[
\Psi_g^2(x) = \theta(\chi(g))(x,1). 
\] 
Since $\Psi_g^2$ sends $x_{i}$ to $x_{i,g}$, $\Sigma(g\circ \Psi_g^2)$ is equal to $\Sigma$. 
We define a mapping $\theta_3:\mathcal{U}_3 \to \mathcal{V}_3$ as $\theta_3(g) = g\circ \Psi_g^2$. 

\begin{lemma}\label{T:continuity theta3}
The set $\mathcal{U}_3$ is an open neighborhood of $f$ in $\mathcal{V}_2$ and $\theta_3$ is continuous (with respect to the topologies $\tau W^s$ for $2\leq s \leq \infty$). 

\end{lemma}

\begin{proof}
By the definition, $\mathcal{U}_3$ is the preimage $\chi^{-1}(\mathcal{O}_N)$ of the open set $\mathcal{O}_N$ and $\Psi_g^2$ is a composition of a proper mapping $N\hookrightarrow N\times I$ defined by $x\mapsto (x,1)$ and $\theta(\chi(g))$. 
By \cite[\S.2, Propositions 1 and 2]{MatherII}, it is enough to show that $\chi$ is continuous. 
It is easy to verify that $\{(U_\alpha\times I, \varphi_\alpha\times \id)\}_{\alpha\in A}$ and $\{(\pi_{TN}^{-1}(U_\alpha), d\varphi_\alpha)\}_{\alpha\in A}$ are systems of coordinate neighborhoods of $N\times I$ and $TN$, respectively (here we identify $T\R^n$ with $\R^{2n}$ in the obvious way), $\{L_\alpha\times I\}_{\alpha\in A}$ is a locally finite covering of $N\times I$ consisting of compact subsets, and $L_\alpha\times I$ is contained in $U_\alpha\times I$ for any $\alpha\in A$.
Moreover, for any $g\in \mathcal{U}_3'$, $\chi(g)(L_\alpha\times I)$ is contained in $\pi_{TN}^{-1}(U_\alpha)$. 
By Theorem~\ref{T:basis Whitney topology}, it is enough to show that for any $g\in \mathcal{U}_3'$ and a system $\varepsilon=\{\varepsilon_{\alpha}\}_{\alpha\in A}$ of positive numbers, there exists a system $\delta$ such that:
\[
\chi(N_s(g,L,\varphi,\id,\delta)) \subset N_s(\chi(g),L\times I,\varphi\times \id, d\varphi, \varepsilon). 
\]
To show this, we first observe that for any $g, h \in \mathcal{U}_3'$ the support of $\chi(h)-\chi(g)$ is contained in $\cup_{i\in \mathcal{I}}\Cl{\varphi_i^{-1}(B(\gamma_i/4))}\times I$, which is a subset of $\cup_{i\in \mathcal{I}}L_i\times I$ and away from $L_\alpha \times I$ for any $\alpha \in A\setminus \mathcal{I}$ by the condition (U6). 
Furthermore, the following holds for any $i\in \mathcal{I}$ and $g, h \in \mathcal{U}_3'$: 
\begin{align*}
&\norm{d\varphi_{i} \circ (\chi(h)-\chi(g)) \circ (\varphi_{i}^{-1}\times \id)}_{s,\varphi_i(L_{i})\times I} \\
=&\left|\varphi_i(x_{i,h}) - \varphi_i(x_{i,g})\right| \cdot \norm{\lambda\circ \varphi_i^{-1}\circ \pi_N}_{s,\varphi_i(L_{i})\times I}. 
\end{align*}
Note that the norm $\norm{\lambda\circ \varphi_i^{-1}\circ \pi_N}_{s,\varphi_i(L_{i})\times I}$ does not depend on $g$ and $h$.
Using Lemma~\ref{T:evaluation critical point/value} to estimate $\left|\varphi_i(x_{i,h}) - \varphi_i(x_{i,g})\right|$, we can take a desired system $\delta$ in a quite similar way to that in the proof of Lemma~\ref{T:continuity theta2}.  
Details are left to the reader. 
\end{proof}

The mapping $\theta_3:\mathcal{U}_3\to\mathcal{V}_3$ satisfies the desired conditions. 
Indeed, $\theta_3(g)=g\circ \Psi_g^2$ is obviously $\mathcal{A}$-equivalent to $g$ for any $g\in \mathcal{U}_3$. 
Since $X_f$ is the zero vector field, the diffeomorphism $\Psi_f^2$ defined by the flow of $X_f$ is the identity mapping, and thus $\theta_3(f)=f$.

\subsection*{Proof of Claim~4}

Here, we will show that there exists an open neighborhood $\mathcal{U}_4\subset \mathcal{V}_3$ of $f$ such that any $g\in \mathcal{U}_4$ is $\mathcal{A}$-equivalent to $f$ (and then put $\theta_4(g) = f$). 
For a vector bundle $E$ over $N$, we define a $C^\infty(N)$-module $\Gamma_k(E)$ ($k\geq 0$) as follows (recall that $\Sigma$ is the critical point set of $f$): 
\[
\Gamma_k(E)=\{\xi\in \Gamma(E)~|~ j^k\xi|_{\Sigma}=0\}.  
\]
Let $\pi_N=\pi:N\times I \to N$ be the projection and 
\[
\Gamma_k(\pi^\ast E)=\{\xi \in \Gamma(\pi^\ast E)~|~ j^k\xi|_{\Sigma\times I}=0\}
\]
for $k\geq 0$, which is a $C^\infty(N\times I)$-module. 
We regard $\Gamma_k(E)$ as a subset of $\Gamma_k(\pi^\ast E)$ via the injection $\pi^\ast:\Gamma_k(E)\to \Gamma_k(\pi^\ast E)$.
Universality of tensor products yields a homomorphism 
\[
\iota_k:\Gamma_k(E)\otimes_{C^\infty(N)} C^\infty(N\times I) \to \Gamma_k(\pi^\ast E). 
\]

\begin{lemma}\label{T:surj iota_i}
The homomorphism $\iota_k$ is surjective. 

\end{lemma}

\begin{proof}
For $k\geq 0$, we define $C^\infty_k(N)$ and $C^\infty_k(N\times I)$ as follows:
\begin{align*}
&C^\infty_k(N) = \{h\in C^\infty(N)~|~ j^kh|_{\Sigma}=0\},\\
&C^\infty_k(N\times I) = \{h\in C^\infty(N\times I)~|~ j^kh|_{\Sigma\times I}=0\}.
\end{align*}
We also regard $C^\infty_k(N)$ as a subset of $C^\infty_k(N\times I)$ via $\pi^\ast$. 
We take a partition of the unity $\{\varrho_i~|~i\in \mathcal{I}\}\cup\{\varrho_V\}$ of $N$ so that they satisfy the following conditions: 
\begin{itemize}

\item 
$\D \varrho_i \equiv 1$ in $\Cl{\varphi_i^{-1}(B(\nu_i/2))}$ and $\D \varrho_i \equiv 0$ outside $\Cl{\varphi_i^{-1}(B(\nu_i))}=\Cl{U_i}$, 

\item 
$\sqrt[n]{\varrho_i}$ is $C^\infty$ for any $n\in \Z_{\geq 0}$. 

\end{itemize}
We denote the function $p_l\circ \varphi_i: L_{i}\to \R$ by $y_l^i$, where $p_l:\R^n\to \R$ be the projection to the $l$-th component. 

We first prove Lemma~\ref{T:surj iota_i} under the assumption that $E$ is a trivial line bundle. 
In this case we can identify $\Gamma_k(E)$ and $\Gamma_k(\pi^\ast E)$ with $C^\infty_k(N)$ and $C^\infty_k(N\times I)$, respectively. 
Let $h\in C^\infty_k(N\times I)$. 
Since $j^k\varrho_V|_\Sigma=0$ for any $k\geq 0$, $\varrho_V$ is an element of $C^\infty_k(N)$ and thus $\varrho_V h = \iota_k(\varrho_V\otimes h) \in \Im(\iota_k)$. 
The support of the function $\varrho_i h$ is contained in $\Cl{U_i}\times I$, so we can regard this function as that on $\R^n\times I$. 
Since the $k$-jet $j^k \varrho_i h|_{\Sigma\times I}$ vanishes, we can decompose this function as follows: 
{\allowdisplaybreaks
\begin{align*}
(\varrho_i h)(x,t) &= (\sqrt[k+2]{\varrho_i}^{k+2}h)(x,t)\\
&=\left(\sqrt[k+2]{\varrho_i}(x)\right)^{k+2}\int_{0}^1\frac{d}{ds}\left(h(sx,t)\right)ds \\
&= \left(\sqrt[k+2]{\varrho_i}(x)\right)^{k+2}\sum_{l_1=1}^n y_{l_1}^i(x)h_{l_1}^i(x,t) \hspace{.3em}\left(h_{l_1}^i(x,t):=\int_{0}^1\frac{\Pa}{\Pa x_{l_1}}\left(h(sx,t)\right)ds\right)\\
&= \left(\sqrt[k+2]{\varrho_i}(x)\right)^{k+2}\sum_{l_1=1}^n y_{l_1}^i(x)\int_0^1 \frac{d}{ds}\left(h_{l_1}^i(sx,t)\right)ds \\
&\vdots \\
&= \sum_{l_1,\ldots,l_{k+1}=1}^n (\sqrt[k+2]{\varrho_i}y_{l_1}^i)(x)\cdots (\sqrt[k+2]{\varrho_i}y_{l_{k+1}}^i)(x)(\sqrt[k+2]{\varrho_i}h_{l_1,\ldots,l_{k+1}}^i)(x,t). 
\end{align*}
}
Since the support of $\sqrt[k+2]{\varrho_i}$ are contained in $\Cl{U_i}\times I$, we can extend the functions $\sqrt[k+2]{\varrho_i}y_{l_j}^i$ and $\sqrt[k+2]{\varrho_i}h_{l_1,\ldots,l_{k+1}}^i$ to those on $N\times I$, which we denote by the same symbols. 
It is easy to see that the function $\sum_{i\in \mathcal{I}} \sqrt[k+2]{\varrho_i}y_{l_1}^i\cdots \sqrt[k+2]{\varrho_i}y_{l_{k+1}}^i$ is contained in $C^\infty_k(N)$. 
We then obtain: 
{\allowdisplaybreaks
\begin{align*}
h =&\varrho_V h +\sum_{i\in \mathcal{I}}\varrho_i h \\
=& \varrho_V h + \sum_{l_1,\ldots,l_{k+1}=1}^n \left(\sum_{i\in \mathcal{I}}\sqrt[k+2]{\varrho_i}y_{l_1}^i\cdots \sqrt[k+2]{\varrho_i}y_{l_{k+1}}^i\right)
\cdot\left(\sum_{i\in \mathcal{I}}\sqrt[k+2]{\varrho_i}h_{l_1,\ldots,l_{k+1}}^i\right) \\
=& \iota_k(\varrho_V\otimes h) \\
&+ \sum_{l_1,\ldots,l_{k+1}=1}^n\iota_k\left(\left(\sum_{i\in \mathcal{I}}\sqrt[k+2]{\varrho_i}y_{l_1}^i\cdots \sqrt[k+2]{\varrho_i}y_{l_{k+1}}^i\right)
\otimes\left(\sum_{i\in \mathcal{I}}\sqrt[k+2]{\varrho_i}h_{l_1,\ldots,l_{k+1}}^i\right) \right).
\end{align*}
}%
Therefore, $h$ is contained in the image of $\iota_k$ and thus $\iota_k$ is surjective. 

We next show the lemma for a general vector bundle $E$. 
If $E$ admits a direct-sum decomposition $E= E_1\oplus E_2$, the homomorphism $\iota_k$ for $E$ is surjective if and only if $\iota_k$ for $E_1$ and $E_2$ are both surjective. 
Any vector bundle $E$ is a direct-sum summand of a trivial bundle, which is a direct-sum of trivial line bundles. 
Since we have already shown the lemma for a trivial line bundle, the statement for $E$ also holds. 
\end{proof}

Although the homomorphism $tf : \Gamma(TN) \to \Gamma(f^\ast T\R)$ is \textit{not} surjective, as its image is contained in $\Gamma_0(f^\ast T\R)$, we can show the following lemma in the same way as that in the proof of \cite[Ch.~III, Proposition 2.2]{GG}: 
\begin{lemma}\label{T:surj tf for Morse}
The mapping $tf : \Gamma_{k-1}(TN) \to \Gamma_k(f^\ast T\R)$ is surjective for any $k\geq 1$. 
\end{lemma}
\noindent
We define $t'f:\Gamma_{k-1}(\pi^\ast TN) \to \Gamma_{k}(\pi^\ast f^\ast T\R)$ as in \cite[\S.7]{MatherII}.

\begin{lemma}\label{T:surj t'f}
The mapping $t'f$ is surjective. 

\end{lemma}

\begin{proof}
We can easily verify that the following diagram commutes:
\[
\begin{CD}
\Gamma_{k-1}(TN)\otimes_{C^\infty(N)} C^\infty(N\times I)  @>\iota_{k-1}>> \Gamma_{k-1}(\pi^\ast TN) \\
@Vtf\otimes \id VV @V t'f VV \\
\Gamma_{k}(f^\ast T\R)\otimes_{C^\infty(N)} C^\infty(N\times I) @>\iota_k>> \Gamma_{k}(\pi^\ast f^\ast T\R)
\end{CD}
\]
Since all the mappings except for $t'f$ are surjective by Lemmas \ref{T:surj iota_i} and \ref{T:surj tf for Morse}, so is $t'f$.  
\end{proof}
\noindent
We define the set $X$ as follows: 
\[
X=\{G\in C^\infty(N\times I)~|~ G_0 = f, \hspace{.3em} \Sigma(G_t)=\Sigma,\hspace{.3em}\Delta(G_t) = \Delta\},
\]
where $G_t\in C^\infty(N\times I)$ is defined by $G_t(x) = G(x,t)$ for $t\in I$. 
Let $x_0= f\circ \pi\in X$ and, following \cite[\S.7]{MatherII}, $\Gamma_k^X(\pi^\ast TN)$ be the set of germs of continuous mappings (with respect to the topologies $\tau W^\infty$) from $X$ to $\Gamma_k(\pi^\ast TN)$ at $x_0$. 
We define $C^X(N\times I)$ and $C^X_k(N\times I)$ in a similar manner. 
Note that $C^X(N\times I)$ is a ring with multiplication induced by that of $C^\infty(N\times I)$, and $C^X_k(N\times I)$ is a $C^X(N\times I)$-module in the obvious way. 
We denote the following set by $\Gamma_k^X(j^\ast T\R)$: 
\[
\left\{\xi:(X,x_0) \to C^\infty(N\times I, T\R)~\left|~\begin{minipage}[c]{58mm}
$\xi$ is continuous with respect to $\tau W^\infty$,

$j^k\xi(G)|_{\Sigma\times I} = 0$,

$\pi_{T\R} \circ \xi(G) = G \mbox{ for any }G\in X$.
\end{minipage}\right.\right\}.
\]
This set is also a $C^X(N\times I)$-module. 
The vector bundle $T\R$ is trivial and naturally identified with $\R\times \R$. 
Under the identification, we can define $p_2:T\R\to \R$ as the projection to the second (fiber) component.
This mapping induces a continuous mapping $p_{2,\ast}:C^\infty(N\times I,T\R)\to C^\infty(N\times I)$. 
It is easy to see that 
\[
(p_{2,\ast})_\ast:\Gamma^X_k(j^\ast T\R) \to C^X_k(N\times I) \mbox{ defined by }(p_{2,\ast})_\ast(\xi) = p_{2,\ast}\circ \xi
\]
is an isomorphism as $C^X(N\times I)$-modules. 
In what follows we will identify them via this isomorphism. 
We also regard $C_k^\infty(N\times I)$ as a subset of $C_k^X(N\times I)$ consisting of constant map-germs. 

\begin{lemma}\label{T:fg C_i^X}
The $C^X(N\times I)$-module $C_k^X(N\times I)$ is finitely generated. 
Furthermore, we can take a finite generating set of $C_k^X(N\times I)$ consisting of elements in $C_k^\infty(N\times I)$. 

\end{lemma}

\begin{proof}
We will prove Lemma~\ref{T:fg C_i^X} by induction on $k$. 
Let $h\in C_0^X(N\times I)$ and $\{\varrho_i~|~i\in \mathcal{I}\}\cup \{\varrho_V\}$ be the partition of the unity given in the proof of Lemma~\ref{T:surj iota_i}. 
By the definition, $\varrho_V$ is contained in $C_0^\infty(N\times I)$.
As in the proof of Lemma~\ref{T:surj iota_i}, we can decompose $h$ as follows: 
\[
\sum_{i\in \mathcal{I}} \varrho_i h = \sum_{l=1}^{n} \left( \sum_{i\in \mathcal{I}}\sqrt{\varrho_i}h_{l}^i \right)\left( \sum_{i\in \mathcal{I}} \sqrt{\varrho_i} y_l^i \right), 
\]
where $\sqrt{\varrho_i}h_l^i$ is a map-germ defined as follows: 
{\allowdisplaybreaks
\begin{align*}
&\sqrt{\varrho_i}h_l^i:(X,x_0) \to C^\infty(N\times I), \\ &\left(\sqrt{\varrho_i}h_l^i(G)\right)(x,t) = \begin{cases}
\sqrt{\varrho_i}(x)\bigints_{\hspace{.25em}0}^{1}\dfrac{\Pa }{\Pa x_l}\left( h(G) (sx,t)\right) ds & (x\in U_{i}) \\
0 & (\mbox{otherwise}).
\end{cases}
\end{align*}
}%
The map-germ $\sum_{i\in \mathcal{I}}\sqrt{\varrho_i}h_l^i$ can be obtained by composing $h$ to the following mappings: 
{\allowdisplaybreaks
\begin{align*}
&C_0^\infty(N\times I) \xrightarrow{\Omega_1} C_0^\infty(\sqcup_{i\in \mathcal{I}}\Cl{U_i}\times I) \xrightarrow{\Omega_2}C_0^\infty(\sqcup_{i\in \mathcal{I}}\Cl{B(\nu_i)}\times I)\\
&\xrightarrow{\Omega_3} C_0^\infty(I\times \sqcup_{i\in \mathcal{I}}\Cl{B(\nu_i)}\times I)\xrightarrow{\Omega_4}C^\infty(I\times \sqcup_{i\in \mathcal{I}}\Cl{B(\nu_i)}\times I)\\
&\xrightarrow{\Omega_5}C^\infty(\sqcup_{i\in \mathcal{I}}\Cl{B(\nu_i)}\times I)\xrightarrow{\Omega_2^{-1}}C_0^\infty(\sqcup_{i\in \mathcal{I}}\Cl{U_i}\times I)\\
&\xrightarrow{\Omega_6}C_\Pa^\infty(\sqcup_{i\in \mathcal{I}}\Cl{U_{i}}\times I)\xrightarrow{\Omega_7}C^\infty(N\times I),
\end{align*}
}%
where, 
\begin{enumerate}

\item
$\Omega_1$ is the restricting mapping, 

\item 
$\Omega_2$ is defined by $\Omega_2(\xi)=\xi\circ (\sqcup_{i\in \mathcal{I}}\varphi_i^{-1})$, where $\sqcup_{i\in \mathcal{I}}\varphi_i^{-1}$ is a diffeomorphism of $\sqcup_{i\in \mathcal{I}}\Cl{B(\nu_i)}$ which is equal to $\varphi_i^{-1}$ on $\Cl{B(\nu_i)}$,

\item 
$\Omega_3$ is defined by $\Omega_3(\xi)=\xi \circ \sigma$, where $\sigma = \sqcup_{i\in \mathcal{I}}\sigma_i$ and $\sigma_i:I\times \Cl{B(\nu_i)}\times I\to \Cl{B(\nu_i)}\times I$ is defined by $\sigma_i(s,x,t) = (sx,t)$ (note that $\sigma$ is proper), 

\item 
$\Omega_4$ is defined by $\Omega_4(\xi)=\frac{\Pa \xi}{\Pa x_l}$, 

\item 
$\Omega_5$ is defined by $\Omega_5(\xi)=\int_0^1\xi ds$, 

\item 
$\Omega_6$ is defined by $\Omega_6(\xi)=(\sqcup_{i\in \mathcal{I}}\sqrt{\varrho_i})\xi=\sum_{i\in \mathcal{I}}\sqrt{\varrho_i}\xi$, where $C_\Pa^\infty(\sqcup_{i\in \mathcal{I}}\Cl{U_i}\times I)$ is a subset of $C^\infty(\sqcup_{i\in \mathcal{I}}\Cl{U_i}\times I)$ consisting of functions which are $0$ on a neighborhood of $\sqcup_{i\in \mathcal{I}}\Pa U_{i}\times I$. 

\item 
$\Omega_7$ is the extension mapping. 

\end{enumerate}
\noindent
We can easily check that these mappings are all continuous (with respect to the topologies $\tau W^\infty$). 
Thus, $\sum_{i\in \mathcal{I}}\sqrt{\varrho_i}h_l^i$ is also continuous, that is, it is an element of $C^X(N\times I)$. 
The element $h$ is eventually decomposed as follows: 
\[
h = h\varrho_V + \sum_{l=1}^{n} \left( \sum_{i\in \mathcal{I}}\sqrt{\varrho_i}h_{l}^i \right)\left( \sum_{i\in \mathcal{I}} \sqrt{\varrho_i} y_l^i \right). 
\]
This decomposition shows that $C^X_0(N\times I)$ is generated by $\varrho_V \in C_0^\infty(N\times I)$ and $\sum_{i\in \mathcal{I}} \sqrt{\varrho_i} y_1^i,\ldots, \sum_{i\in \mathcal{I}} \sqrt{\varrho_i} y_n^i\in C_0^\infty(N\times I)$. 

For $h\in C_k^X(N\times I)$ with general $k$, we can obtain the following decomposition in the same way as above: 
\[
h = \varrho_V h + \sum_{l=1}^{n} \left( \sum_{i\in \mathcal{I}}\sqrt{\varrho_i} h_l^i \right)\left( \sum_{i\in \mathcal{I}} \sqrt{\varrho_i} y_l^i \right).  
\]
The function $\varrho_V$ is contained in $C_k^\infty(N\times I)$. 
Moreover, we can deduce from construction of $\sqrt{\varrho_i}h_l^i$ that the map-germ $\sum_{i\in \mathcal{I}}\sqrt{\varrho_i} h_l^i $ is contained in $C^X_{k-1}(N\times I)$, which has a finite generating set $g_1,\ldots, g_m\in C^\infty_{k-1}(N\times I)$ for some $m\in \N$ by the induction hypothesis. 
Thus, the decomposition of $h$ above shows that $C_k^X(N\times I)$ is generated by $\varrho_V$ and $g_k \left( \sum_{i\in \mathcal{I}} \sqrt{\varrho_i} y_l^i \right)$'s ($k=1,\ldots, m$ and $l=1,\ldots, n$), which are contained in $C_k^\infty(N\times I)$. 
\end{proof}

We define $\mathrm{Ev}:C^X(N\times I)\to C^\infty(N\times I)$ and $\mathrm{Ev}_k:C_k^X(N\times I)\to C_k^\infty(N\times I)$ by evaluating a map-germ at $x_0$. 

\begin{lemma}\label{T:relation kernel evaluation}
For any $k\geq 0$, $\Ker(\mathrm{Ev}_k) = \Ker(\mathrm{Ev})\cdot C_k^X(N\times I)$. 

\end{lemma}

\begin{proof}
It is clear that the module $\Ker(\mathrm{Ev})\cdot C_k^X(N\times I)$ is contained in $\Ker(\mathrm{Ev}_k)$. 
On the other hand, by applying the procedure in the proof of Lemma~\ref{T:fg C_i^X} for decomposing a map-germ to $h \in \Ker(\mathrm{Ev}_k)$, we can decompose $h$ into a sum of elements in $\Ker(\mathrm{Ev})\cdot C_k^X(N\times I)$. 
\end{proof}

We define a $C^X(N\times I)$-module homomorphism $t\overline{j}:\Gamma_{k-1}^X(\pi^\ast TN) \to \Gamma_k^X(j^\ast T\R)$ as follows: 
\[
t\overline{j}([\xi])(g) = t'x (\xi(g))\mbox{ for }g\in X. 
\]

\begin{lemma}\label{T:surj toverlinej}
The mapping $t\overline{j}$ is surjective and the following equality holds:
\[
\Ker\left(\mathrm{Ev}_k:\Gamma_k^X(j^\ast T\R)\to \Gamma_k((f\circ \pi)^\ast T\R)\right) =t\overline{j}\left(\Ker(\mathrm{Ev})\cdot \Gamma_{k-1}^X(\pi^\ast TN)\right).
\]
 
\end{lemma}

\begin{proof}
For surjectivity of $t\overline{j}$, it is sufficient to see that a generating set of $\Gamma_k^X(j^\ast T\R)$ (as a $C^X(N\times I)$-module) is contained in the image $t\overline{j}(\Gamma_{k-1}^X(\pi^\ast TN))$ since $t\overline{j}$ is a $C^X(N\times I)$-module homomorphism. 
The following diagram commutes: 
\[
\xymatrix{
\Gamma_{k-1}^X(\pi^\ast TN) \ar[r]^{t\overline{j}} & \Gamma_k^X(j^\ast T\R) \\
\Gamma_{k-1}(\pi^\ast TN) \ar[u]\ar[r]^{t'f} & \Gamma_k((f\circ \pi)^\ast T\R), \ar[u]
}
\]
where the vertical arrows are the inclusion mappings (note that we regard $\Gamma_l(E)$ as a subset of $\Gamma_l^X(E)$ consisting of constant map-germs).
Since the mapping $t'f$ is surjective by Lemma~\ref{T:surj t'f}, the image of $t\overline{j}$ contains $\Gamma_k((f\circ \pi)^\ast T\R) \subset \Gamma_k^X(j^\ast T\R)$, in which we can take a generating set of $\Gamma_k^X(j^\ast T\R)$ by Lemma~\ref{T:fg C_i^X}. 

As for the statement on the kernel, we can first obtain $\Ker(\mathrm{Ev}_k) = \Ker(\mathrm{Ev})\cdot \Gamma_k^X(j^\ast T\R)$ by Lemma~\ref{T:relation kernel evaluation}. 
Since $\Gamma_k^X(j^\ast T\R)$ is equal to $t\overline{j}(\Gamma_{k-1}^X(\pi^\ast TN))$, we obtain:
\[
\Ker(\mathrm{Ev}_k) = \Ker(\mathrm{Ev})\cdot t\overline{j}(\Gamma_{k-1}^X(\pi^\ast TN)) = t\overline{j}\left(\Ker(\mathrm{Ev})\Gamma_{k-1}^X(\pi^\ast TN)\right). 
\]
\end{proof}

For $\alpha \in A$, we define $\mu_\alpha \in \R$ as follows: 
\[
\mu_\alpha = \min_{x\in L_\alpha}\norm{D^1f(x)}. 
\]
By the condition (U6), $f$ has no critical point in $L_\alpha$, in particular $\mu_\alpha >0$ for any $\alpha \in A\setminus \mathcal{I}$. 
We define a system $\tilde{e} = \{\tilde{e}_\alpha\}_{\alpha \in A}$ as follows:
\[
\tilde{e}_\alpha = \begin{cases}
e_\alpha = \dfrac{\gamma_i}{4n} & (\alpha = i\in \mathcal{I})\\[5pt]
\dfrac{\mu_\alpha}{2} & (\alpha \not\in \mathcal{I}).
\end{cases}
\]
Let $\mathcal{U}_4'=N_2(f,L,\varphi,\id,\tilde{e})\cap \mathcal{V}_3$. 
For $g\in\mathcal{U}_4'$, we define $\overline{g}\in C^\infty(N\times I)$ as follows: 
\[
\overline{g}:N\times I \to \R, \hspace{.5em}\overline{g}(x,t) = tg(x) + (1-t) f(x). 
\]
This function is an element of $X$. 
Indeed, it is obvious that $g_0$ is equal to $f$. 
Since $dg_t = df + t (dg-df)$ and $\Sigma(g) = \Sigma$, $\Sigma$ is contained in $\Sigma(g_t)$ for any $t\in I$. 
Moreover, the following inequalities hold for any $t\in I$, $\alpha\in A$ and $x\in L_\alpha$: 
\begin{align*}
&\norm{D^1g_t(x)} \geq \norm{D^1f(x)} - t\norm{D^1(g-f)(x)}\geq \norm{D^1f(x)} - \tilde{e}_\alpha, \\
&\norm{g_t-f}_{2,L_i} \leq \norm{g-f}_{2,L_i} < \frac{\nu_i}{2n}. 
\end{align*}
The first inequality implies $\norm{D^1g_t(x)}$ is greater than $\mu_\alpha/2>0$ for $\alpha \in A\setminus \mathcal{I}$. 
Thus, there exist no critical points in $L_\alpha$. 
On the other hand, as we observed in the beginning of the proof of Claim 2, we can deduce from the two inequalities above that for each $i\in \mathcal{I}$ there exists exactly one critical point of $g$ in $\varphi_i^{-1}(B(\gamma_i/8))\subset L_i$ (which is $x_i$ as $\Sigma \subset \Sigma(g_t)$). 
We can therefore conclude that $\Sigma(g_t) = \Sigma$. 
The critical value set $\Delta(g_t)$ is also equal to $\Delta$ as $\Delta(g) = \Delta$. 

We define the mapping $\theta_4':\mathcal{U}_4'\to X$ by $\theta_4'(g)= \overline{g}$.
This function is continuous with respect to the topology $\tau W^\infty$ (cf.~\cite[\S.2, Propositions 2 and 3]{MatherII}). 
We further define a mapping $\xi:X\to C^\infty(N\times I,T\R)$ as follows: 
\[
\xi(G)(x,t) = \frac{d}{dt}\left(G(x,t) \right).
\]
It is easy to check that $\xi$ is continuous and $\pi_{T\R} \circ \xi(G) =G$ for any $G\in X$. 
Moreover, since $G(x_i,t) = y_i$ and $\frac{\Pa (G\circ (\varphi_i^{-1}\times \id))}{\Pa x_m}(0,t)=0$ for any $G\in X$, $t\in I$, $x_i\in \Sigma$ and $m=1,\ldots, n$, $j^1\xi(G)|_{\Sigma\times I}$ is equal to $0$ for any $G\in X$. 
Thus, $\xi$ represents an element in $\Gamma^X_1(j^\ast T\R)$. 
Since $\xi(x_0)$ is equal to $0$, $[\xi]$ is an element in $\Ker(\mathrm{Ev}_1)$. 
By Lemma~\ref{T:surj toverlinej}, there exists $\zeta\in \Ker(\mathrm{Ev})\cdot \Gamma_0^X(\pi^\ast TN)$ such that $[\xi]=t\overline{j}(\zeta)$. 
Let $\tilde{\zeta}$ be a representative of $\zeta$ (which is a mapping from a neighborhood of $x_0$ in $X$ to $\Ker(\mathrm{Ev})\cdot \Gamma_0(\pi^\ast TN)$), and $\mathcal{O}_N\subset \Gamma(\pi^\ast TN)$ be a neighborhood of the zero-section we can take by Lemma~\ref{T:continuity integral vectfield}.
Since $\tilde{\zeta}(x_0) = 0$, we can take an open neighborhood $\mathcal{U}_4''\subset X$ of $x_0$ so that the image $\tilde{\zeta}(\mathcal{U}_4'')$ is contained in $\mathcal{O}_N$. 
Let $\mathcal{U}_4 = (\theta_4')^{-1}(\mathcal{U}_4'')$ and we define a diffeomorphism $\Psi_g^3:N\to N$ for for $g\in \mathcal{U}_4$ as follows: 
\[
\Psi_g^3(x)=\theta\left(\tilde{\zeta}(\theta_4'(g))\right)(x,1).
\] 
We can verify that $g\circ \Psi_g^3$ is equal to $f$ for any $g\in \mathcal{U}_4$ (cf.~\cite[\S.7]{MatherII}), and we eventually complete the proof of Claim 4, and (1) of Theorem~\ref{T:suff condi stability}.

In summary, for any $g\in \mathcal{U}=(\theta_3\circ \theta_2\circ \theta_1)^{-1}(\mathcal{U}_4)$, we have constructed self-diffeomorphisms 
\[
\beta(g), \Psi_{\theta_1(g)}^1, \Psi_{\theta_2\circ \theta_1(g)}^2, \Psi_{\theta_3\circ \theta_2,\theta_1(g)}^3:N\to N
\] and a self-diffeomorphism $\psi_{\theta_1(g)}:\R\to \R$ satisfying the following: 
\[
\psi_{\theta_1(g)}^{-1}\circ g\circ \beta(g)\circ \Psi_{\theta_1(g)}^1 \circ \Psi_{\theta_2\circ \theta_1(g)}^2  \circ \Psi_{\theta_3\circ \theta_2,\theta_1(g)}^3= f. 
\]
We denote the diffeomorphisms $\beta(g)\circ \Psi_{\theta_1(g)}^1 \circ \Psi_{\theta_2\circ \theta_1(g)}^2  \circ \Psi_{\theta_3\circ \theta_2,\theta_1(g)}^3$ and $\psi_{\theta_1(g)}^{-1}$ by $\beta_s(g)$ and $\beta_t(g)$, respectively. 
As shown in the proof of Lemma~\ref{T:continuity theta2}, the mapping $\beta_t:\mathcal{U}\to \Diff(\R)$ is continuous. 
We have also verified that the mappings $\beta$, $g\mapsto \Psi_{\theta_2\circ \theta_1(g)}^2$, and $g\mapsto\Psi_{\theta_3\circ \theta_2,\theta_1(g)}^3$ are continuous. 
Moreover, $\Psi_{\theta_1(g)}^1$ becomes the identity when $f$ is quasi-proper (cf.~Remark~\ref{Re:case quasi-proper}). 
Therefore, the mapping $\beta_s$ is also continuous provided that $f$ is quasi-proper. 
This completes the proof of (2) of Theorem~\ref{T:suff condi stability}.

\section{Applications}\label{Se:application}

In this section we will give two applications of Theorem~\ref{T:suff condi stability}. 
We first give an explicit example of strongly but not infinitesimally stable function. 
As we noted in the introduction, we could not obtain such an example relying on known results. 
We then discuss stability of Nash functions, relating it with behavior of their gradients around the end of the source spaces.

\subsection{A strongly stable but not infinitesimally stable function}

In this subsection we will prove the following theorem:

\begin{theorem}\label{T:ex strongly but not inf stable func}
The function $F:\R\to \R$ defined by $F(x) = \exp(-x^2)\sin x$ is strongly stable but not infinitesimally stable. 

\end{theorem}

\begin{proof}
As we briefly reviewed in Subsection~\ref{Se:stability maps}, a locally stable (Morse) function $f:\R\to \R$ is infinitesimally stable if and only if $f|_{\Sigma(f)}$ is proper, which is equivalent to the condition $Z(f|_{\Sigma(f)})=\emptyset$, where $Z(f|_{\Sigma(f)})$ was defined in the beginning of Section~\ref{Se:preliminaries}.
It is also known that quasi-properness of a function $f$ is equivalent to the condition $Z(f)\cap \Delta(f)=\emptyset$.
Since a locally stable function $f$ is strongly stable if and only if $f$ is quasi-proper by Theorem~\ref{T:suff condi stability}, it is enough to show the followings: 
\begin{enumerate}

\item 
$F$ is a Morse function, 

\item 
$Z(F) = Z(F|_{\Sigma(F)}) = \{0\}$, 

\item 
$\Delta(F)$ does not contain $0$. 

\end{enumerate}
To see them, we need the first and the second derivatives of $F$, which can be calculated as follows:
{\allowdisplaybreaks
\begin{align*}
F'(x) &= \exp(-x^2)(-2x \sin x +\cos x), \\
F''(x) &= \exp(-x^2) \left\{(4x^2-3)\sin x -4x \cos x\right\}. 
\end{align*}
}%
Thus $F'(x)$ is equal to $0$ if and only if $\tan x$ is equal to $1/2x$. 
Let $a_n\in \R$ be the $n$-th smallest positive solution of the equation $\tan x = 1/2x$. 
It is easy to see that $\Sigma(F)$ is equal to $\{\pm a_n\in \R~|~n\in \Z_{>0}\}$. 
We can further verify the following properties of the sequence $\{a_n\}_{n\geq 1}$: 

\begin{enumerate}[(A)]

\item 
$\D n\pi < a_n < \frac{(2n+1)\pi}{2}$, in particular $\D \lim_{n\to \infty} a_n = \infty$, 

\item 
$\D \left|\sin a_n\right| >\left|\sin a_{n+1}\right| >0$ for any $n>0$. 

\end{enumerate}
We can deduce from the condition (B) that $\left|F(\pm a_n)\right| = \left|F(\pm a_m)\right|$ if and only if $n=m$. 
Since $F$ is an odd function, we can conclude that $F|_{\Sigma}$ is injective. 
Suppose that $F''(a_n)$ were equal to $0$ for some $n>0$. 
The solution $a_n$ would satisfy the equality $\D \frac{4a_n}{4a_n^2-3}=\frac{1}{2a_n}$, but the equation $\D \frac{4x}{4x^2-3}=\frac{1}{2x}$ does not have a real solution. 
Hence each critical point of $F$ is non-degenerate, concluding that $F$ is a Morse function. 

We can deduce from the condition (A) on $\{a_n\}_{n>0}$ that $0$ is contained in $Z(F|_{\Sigma})$ (and also in $Z(F)$). 
On the other hand, since $\D \lim_{x\to \pm \infty}F(x)$ is equal to $0$, $0\in \R$ is the only improper point of $F$ (and that of $F|_{\Sigma}$).
Lastly, by the condition (B) we can prove that $\Delta(F)$ does not contain $0$. 
\end{proof}

\subsection{Stability of Nash functions}
In this subsection we will discuss stability of Nash functions.
The reader can refer to \cite{BochnakCosteRoy}, for example, for the definition and basic properties of Nash functions.  

As mentioned in the introduction, the complement of $\tau(f)$, denote by $B_\infty(f)$, is called the set of \textit{bifurcation values at infinity} for a semi-algebraic mapping $f$. 
We first introduce two conditions related to bifurcation values. 
Let $f:\R^n\to \R$ be a Nash function. 
We say that $f$ satisfies the \textit{Fedoryuk} (resp.~\textit{Malgrange}) \textit{condition} at $y\in \R$ if there exist an open neighborhood $V$ of the end of $\R^n$ and positive numbers $\delta, \varepsilon>0$ such that $\norm{\nabla f(x)}>\varepsilon$ (resp.~$\norm{x}\cdot\norm{\nabla f(x)}>\varepsilon$) for any $x\in f^{-1}(y-\delta,y+\delta)\cap V$, where $\nabla f:\R^n\to \R^n$ is the gradient of $f$. 
It immediately follows from the definitions that $f$ satisfies the Malgrange condition at $y$ if it satisfies the Fedoryuk condition at $y$.
However, the converse does not hold in general (see e.g.~\cite{KOSSemialgSard}). 

\begin{proposition}\label{T:malgrange imply end-trivial}

A Nash function $f:\R^n\to \R$ is end-trivial at $y\in \R$ if $f$ satisfies the Malgrange condition at $y\in \R$. 
The trivializing mapping $\Phi$ in the definition of end-triviality can be given by the flow of $\nabla f/\norm{\nabla f}^2$. 

\end{proposition}

\noindent
Although this proposition is regarded as a well-known fact in the literature, apparently no references give its proof explicitly. 
The proposition can be shown as follows:
By the assumption, no critical points of $f$ are contained in $f^{-1}(y-\delta,y+\delta)\setminus B(R)$ for some $R,\delta>0$. 
Furthermore, replacing $R$ (resp.~$\delta$) with larger (resp.~smaller) one if necessary, we can verify (in a way similar to the proof of \cite[Theorem 3.5]{DGgradient}) that the flow of $\nabla f/\norm{\nabla f}^2$ gives an embedding $\Phi:\left(f^{-1}(y)\setminus \Cl{B(R)}\right)\times [y-\delta,y+\delta] \to f^{-1}([y-\delta,y+\delta])$ such that the complement of its image is compact, which we denote by $K$.
Thus, $\Phi$ is the trivializing mapping for a restriction of $f$ on $f^{-1}([y-\delta,y+\delta])\cap V$, where $V=\R^n\setminus K$. 

We can immediately deduce the following from Theorem~\ref{T:suff condi stability} and Proposition~\ref{T:malgrange imply end-trivial}: 

\begin{corollary}\label{T:suff condi stability semi-algebraic mapping}

A locally stable Nash function $f:\R^n\to \R$ is stable if it satisfies the Malgrange condition at any critical value of $f$. 

\end{corollary}

\begin{example}\label{Ex:modelMorse stable}
For $k\in \{0,\ldots,n\}$ we define a function $G_k:\R^n\to \R$ as follows: 
\[
G_k(x_1,\ldots,x_n) = \sum_{i=1}^{k} x_i^2 -\sum_{j=k+1}^n x_j^2. 
\]
It is obvious that $G_k$ is a Morse function (locally stable) for any $k$. 
Since the functions $G_0$ and $G_n$ is proper, we can deduce from the result of Mather \cite{MatherV} that they are (strongly) stable. 
However, $G_k$ with $0<k<n$ is \textit{not} quasi-proper (especially not strongly stable). 
The function $G_k$ is defined by a polynomial, in particular it is a Nash function. 
Since the gradient $\nabla G_k$ is a diffeomorphism, $G_k$ satisfies the Malgrange (and even Fedoryuk) condition at any value in $\R$. 
We can therefore deduce from Corollary~\ref{T:suff condi stability semi-algebraic mapping} that $G_k$ is stable. 
\end{example}

\begin{corollary}\label{T:stable after linear perturbation}
Let $f:\R^n\to \R$ be a Nash function. 
There exists a Lebesgue measure zero subset $\Sigma\subset \R^n$ such that for any $a=(a_1,\ldots, a_n) \in \R^n\setminus \Sigma$ the function $f_a:\R^n\to \R$ defined as
\[
f_a(x_1,\ldots,x_n) = f(x_1,\ldots,x_n) + \sum_{i=1}^{n}a_ix_i 
\]
is stable. 

\end{corollary}


\begin{proof}
We can first deduce from \cite[Theorem 2]{Ichiki} that there exists a Lebesgue measure zero subset $\Sigma_1\subset \R^n$ such that $f_a$ is locally stable for any $a\in \R^n\setminus \Sigma_1$. 
Since the gradient $\nabla f$ is a Nash mapping, we can take finitely many semi-algebraic subsets $T_1,\ldots, T_l \subset \R^n$ with $\R^n = \bigcup_{i} T_i$ so that for each $i=1,\ldots,l$ there exist a semi-algebraic set $F_i$ and a semi-algebraic homeomorphism $\theta_i:(\nabla f)^{-1}(T_i)\to F_i\times T_i$ such that $\nabla f|_{(\nabla f)^{-1}(T_i)} = p_2\circ \theta_i$ (see \cite[Theorem 9.3.2]{BochnakCosteRoy}).
We define a subset $\Sigma_2\subset \R^n$ as follows: 
\[
\Sigma_2 = \bigcup_{\dim F_j \geq 1} \{-a\in \R^n~|~ a\in T_j\}. 
\]
Since the dimension of $F_i\times T_i$ (as a semi-algebraic set) is at most $n$, the dimension of $\Sigma_2$ is less than $n$. 
Thus $\Sigma_2$ has Lebesgue measure zero. 
Let $\Sigma$ be the union $\Sigma_1\cup \Sigma_2$, which has Lebesgue measure zero. 
For any $a\in \R^n\setminus \Sigma$, the function $f_a$ is locally stable. 
Furthermore, there exists $\varepsilon>0$ such that $(\nabla f_a)^{-1}(\Cl{B(\varepsilon)}) = (\nabla f)^{-1}(-a + \Cl{B(\varepsilon)})$ is compact, and thus $f_a$ satisfies the Fedoryuk condition at any value in $\R$. 
We can then deduce from Corollary~\ref{T:suff condi stability semi-algebraic mapping} that $f_a$ is stable for any $a\in \R^n\setminus \Sigma$. 
\end{proof}

\appendix

\section{Estimates on critical points/values of Morse functions}

Recall that we denote by $B(r)\subset \R^n$ the open $n$-ball with radius $r$ centered at the origin.
We define a Morse function $f:B(r)\to \R$ as follows: 
\[
f(x_1,\ldots,x_n) = \sum_{i=1}^{n}(-1)^{\varepsilon_i}x_i^2+c,
\]
where $\varepsilon_i = 0$ or $1$. 
In this appendix we will obtain several estimates on configuration of critical points and values of functions close to $f$. 
The estimates below are used repeatedly in Section~\ref{Se:proof main thm}.

\begin{lemma}\label{T:crit pt unique}
Suppose that $r$ is less than $1$.
If a function $g:B(r)\to \R$ satisfies the inequality $\norm{g-f}_{2,B(r)} < r/n$, there exists a unique critical point of $g$ in $B(r)$. 

\end{lemma}

\begin{proof}
In what follows we will regard sections of $T^\ast B(r)$ (such as $df$ and $dg$) as smooth mappings from $B(r)$ to $\R^n$ in the obvious way. 
Note that $df(x)$ is equal to $(2(-1)^{\varepsilon_1}x_1,\ldots,2(-1)^{\varepsilon_n}x_n)$.
We define a smooth mapping $F:B(r)\to \R^n$ by $F(x) = x-(df)^{-1}\circ dg(x)$.
We can easily deduce from the assumption that the absolute value of each partial derivative of $F$ in $B(r)$ is less than $r/2n$. 
For any $x,x'\in B(r)$, the norm $\norm{F(x)-F(x')}$ can be estimated as follows: 
\begin{equation}\label{Eq:estimate norm for Morse}
\begin{split}
&\norm{F(x) - F(x')} \\
\leq& \sqrt{n} \max_{i=1,\ldots,n} \left|F_i(x) - F_i(x')\right| \\
\leq& \sqrt{n} \max_{i=1,\ldots,n} \int_{0}^{1} \left|\frac{d}{dt} F_i(tx +(1-t)x') \right|dt \\
\leq& \sqrt{n} \max_{i=1,\ldots,n} \int_{0}^{1} \sum_{j=1}^n\left|\frac{\Pa F_i}{\Pa x_j}(tx+(1-t)x')\right|\cdot \left|x_{j}-x_{j}'\right|dt \\
\leq& \sqrt{n} \frac{r}{2n} \sum_{j=1}^n \left|x_{j}-x_{j}' \right|\leq \frac{r}{2}\norm{x-x'}. 
\end{split}
\end{equation}

We inductively define $x_k\in \R^n$ as follows:
\[
x_0 =0, \hspace{.5em}x_{k+1} = F(x_k). 
\]
The norm $\norm{x_1}$ ($=\norm{x_1-x_0}$) can be evaluated as follows: 
{\allowdisplaybreaks
\begin{align*}
\norm{x_1} &=\norm{(df)^{-1}\circ dg (0)}=\frac{1}{2}\norm{dg(0)-df(0)}\leq \frac{1}{2}\norm{g-f}_{2,B(1)} < \frac{r}{2n} \leq \frac{r}{2}.  
\end{align*}
}%
Thus, we can deduce the following inequality from \eqref{Eq:estimate norm for Morse} by induction on $k$: 
\begin{equation}\label{Eq:estimate norm}
\norm{x_k-x_{k-1}} < \left(\frac{r}{2}\right)^k.
\end{equation}
In particular, $\{x_k\}_{k=1}^\infty$ has a limit point $x\in \R^n$ since $\{x_k\}_{k=1}^\infty$ is a Cauchy sequence. 
By the definition of $x_k$, $dg(x)$ is equal to $df(0)=0$. 
Thus $x$ is a critical point of $g$. 
Furthermore, as $r<1$, the norm $\norm{x}$ is less than or equal to 
\[
\sum_{k=1}^\infty \norm{x_{k} - x_{k-1}} < \sum_{k=1}^\infty \left(\frac{r}{2}\right)^{k}=\frac{r}{2-r}<r. 
\]
Hence $x$ is contained in $B(r)$. 

For points $x,x'\in B(r)$, the norm $\norm{df^{-1}\circ dg(x) - df^{-1}\circ dg(x')}$ can be estimated as follows: 
\begin{align*}
&\norm{df^{-1}\circ dg(x) - df^{-1}\circ dg(x')} \geq \norm{x-x'} - \norm{F(x)-F(x')}\\
\geq &\left(1-\frac{r}{2}\right)\norm{x-x'} > \frac{1}{2}\norm{x-x'}. 
\end{align*}
Thus, the mapping $dg$ is injective on $B(r)$, in particular a critical point of $g$ in $B(r)$ is unique. 
\end{proof}

\begin{lemma}\label{T:evaluation critical point/value}
Suppose that functions $g,h:B(r)\to \R$ satisfy the inequalities $\norm{g-f}_{2,B(r)}<r/n$ and $\norm{h-f}_{2,B(r)} < r/n$. 
Let $x_{g},x_{h}\in B(r)$ be critical points of $g, h$ in $B(r)$ (existence of such points follows from Lemma~\ref{T:crit pt unique}) and $y_{g},y_{h}$ their images. 
If the points $x_g,x_h$ are contained in a convex set $U\subset B(r)$, the following inequalities hold: 
{\allowdisplaybreaks
\begin{align*}
\left|x_g-x_{h}\right|  < & \sqrt{n}\norm{g-h}_{1,U}, \\
\left|y_g-y_h\right| < & \left(\sqrt{n}\norm{g}_{1,U} +1\right) \norm{g-h}_{1,U}. 
\end{align*}
}

\end{lemma}

\begin{proof}
By the assumption the following inequalities hold for any $j,k\in \{1,\ldots,n\}$ ($j\neq k$) and $x\in B(r)$:
\[
\left|\frac{\Pa^2 g}{\Pa x_j^2}(x)\right|> 2-\frac{r}{n}, \hspace{.5em} \left|\frac{\Pa^2 g}{\Pa x_j\Pa x_k}(x)\right|< \frac{r}{n}. 
\]
For $z\in U$, the norm $\norm{D^1h(z)}$ can be estimated as follows: 
\[
\norm{D^1h(z)} 
\geq\norm{D^1g(z)}- \norm{D^1h(z) - D^1g(z)} 
\geq\norm{D^1g(z)}- \norm{g-h}_{1,U}.
\]
Let $\D m_z = \max_{j} \left|z_j-(x_{g})_j\right|$ and we take $j_0 \in \{1,\ldots, n\}$ so that $m_z$ is equal to $\left|z_{j_0}-(x_{g})_{j_0}\right|$. 
We can then estimate the norm $\norm{D^1g(z)}$ as follows: 
{\allowdisplaybreaks
\begin{align*}
\norm{D^1g(z)} = & \norm{D^1g(z) - D^1g(x_{g})}\\
\geq & \left| \frac{\Pa g}{\Pa x_{j_0}}(z) - \frac{\Pa g}{\Pa x_{j_0}}(x_{g})\right|\\
= & \left|\sum_{k=1}^n (z_k - {(x_g)}_k)\frac{\Pa^2 g}{\Pa x_k \Pa x_{j_0}}(cz +(1-c)x_g) \right| & (\mbox{for }\exists c\in [0,1])\\
\geq & m_z \left(\min_{w\in U}\left|\frac{\Pa^2 g}{\Pa x_{j_0}^2}(w)\right|-\sum_{j\neq j_0} \max_{w\in U}\left|\frac{\Pa^2 g}{\Pa x_j\Pa x_{j_0}}(w)\right|\right) \\
> & m_z \left(2 - n \frac{r}{n}\right) > m_z \geq \frac{\left|z-x_{g}\right|}{\sqrt{n}}. 
\end{align*}
}
Hence, $\D \norm{D^1h(z)}$ is greater than $\D \frac{\left|z-x_{g}\right|}{\sqrt{n}}-\norm{g-h}_{1,U}$. 
Since the point $x_{h}$ is a critical point of $h$, $D^1h(x_{h})$ is equal to $0$, in particular $\D \frac{\left|x_{h}-x_{g}\right|}{\sqrt{n}}-\norm{g-h}_{1,U}$ is less than $0$.
We thus obtain the inequality $\left|x_{h} - x_{g}\right| < \sqrt{n}\norm{g-h}_{1,U}$. 
The estimate of the norm $\left|y_g-y_h\right|$ can be obtained as follows:
{\allowdisplaybreaks
\begin{align*}
|y_{g}-y_{h}| & = |g(x_g)-h(x_{h})| \\
& \leq |g(x_g)-g(x_{h})| + |g(x_{h})-h(x_{h})|\\
& \leq \left(\max_{x\in U}\norm{dg_x} \right) \cdot \left|x_g-x_{h}\right| + \norm{g-h}_{1,U} \\
& < \left(\sqrt{n}\norm{g}_{1,U} +1\right) \norm{g-h}_{1,U}.
\end{align*}
}
\end{proof}

\noindent
{\bf Acknowledgments.}
The author would like to express his gratitude to Takashi Nishimura and Shunsuke Ichiki for 
helpful discussions throughout the course of this work. 
The author would also like to thank Maria Michalska for informing the author of the notion of bifurcation values and relevant references. 
The author was supported by JSPS KAKENHI (Grant Number 17K14194).



\begin{thebibliography}{99}
%
%
\bibitem{BochnakCosteRoy}
J.~Bochnak, M.~Coste and M.-F.~Roy, \textit{Real algebraic geometry}, Ergebnisse der Mathematik und ihrer Grenzgebiete (3), 36. Springer-Verlag, Berlin, 1998. x+430 pp.

\bibitem{DGgradient}
D.~D'Acunto and V.~Grandjean, \textit{On gradient at infinity of semialgebraic functions},  Ann. Polon. Math. \textbf{87}(2005), 39--49. 


\bibitem{DTDetectingBifValue}
{L.~R.~G.~Dias and M.~Tib\u{a}r},
\textit{Detecting bifurcation values at infinity of real polynomials},
{Math. Z.}
\textbf{279}{(2015)},
{311--319}.

\bibitem{Dimca}
A.~Dimca, \textit{Morse functions and stable mappings}, 
Rev. Roumaine Math. Pures Appl. 24 (1979), no. 9, 1293--1297. 

\bibitem{duPlessisVosegaard}
A.~du Plessis and H.~Vosegaard, 
\textit{Characterisation of strong smooth stability}, 
Math.~Scand, \textbf{88}(2001), 193--228. 

\bibitem{duPlessisWall}
A.~du Plessis and T.~Wall,
\textit{The geometry of topological stability},
London Mathematical Society Monographs. New Series, 9. Oxford Science Publications. The Clarendon Press, Oxford University Press, New York, 1995. viii+572 pp.

\bibitem{Engelking}
R.~Engelking, 
\textit{General topology}, 
Sigma Series in Pure Mathematics, \textbf{6}(1989), vii+529 pp.

\bibitem{GG}
M.~Golubitsky and V.~Guillemin, \textit{Stable mappings and their singularities}, Graduate Texts in Mathematics, Vol. 14. Springer-Verlag, New York-Heidelberg, 1973. x+209 pp.

\bibitem{GKtrisection}
D.~Gay and R.~Kirby, {\it Trisecting $4$-manifolds}, Geom. Topol., \textbf{20}(2016), 3097--3132.

\bibitem{Ichiki}
S.~Ichiki, \textit{Generic linear perturbations}, preprint (to appear in Proc.~Amer.~Math.~Soc.), available at arXiv:1607.03220. 


\bibitem{KMPGradConj}
K.~Kurdyka, T.~Mostowski and A.~Parusi\'{n}ski, \textit{Proof of the gradient conjecture of R. Thom}, Ann.~of Math. (2), \textbf{152}(2000), no.~3, 763--792.

\bibitem{KOSSemialgSard}
K.~Kurdyka, P.~Orro and S.~Simon, \textit{Semialgebraic Sard theorem for generalized critical values}, J.~Differential Geom. \textbf{56}(2000), no.~1, 67--92. 

\bibitem{MatherII}
J.~N.~Mather, \textit{Stability of $C^{\infty }$ mappings. II. Infinitesimal stability implies stability},
Ann.~of Math. (2) 89(1969), 254--291. 

\bibitem{MatherV}
J.~N.~Mather, 
\textit{Stability of $C^{\infty }$ mappings. V. Transversality}, 
Advances in Math. \textbf{4}(1970), 301--336. 

\bibitem{SaekiYamamoto}
O.~Saeki and T.~Yamamoto, \textit{Singular fibers of stable maps and signatures of 4-manifolds}, Geom.~Topol. \textbf{10}(2006), 359--399.

\bibitem{Whitney}
H.~Whitney, \textit{Singularities of mappings of Euclidean spaces}, International symposium on algebraic topology, 1958, 285--301.

%
%
%
%
%
%
%
%
%
%
%
%
%

\end{thebibliography}
\end{document}